\documentclass[12pt,reqno]{amsart}
\usepackage{moreverb}

\usepackage{epsfig,color}
\usepackage{amsmath,amssymb,amsfonts}
\usepackage{subfigure,psfrag}
\usepackage{ifthen}
\usepackage{epstopdf}
\usepackage{fullpage}

\definecolor{darkblue}{rgb}{0,0,.8}
\definecolor{darkred}{rgb}{0.8,0,0}

\usepackage{hyperref}
\hypersetup{
  pdftex=true,
  colorlinks=true,
  allcolors = darkred,
  linkcolor=darkblue, 
  menucolor=darkblue,
  urlcolor = darkblue,
  pagecolor=darkblue,
  frenchlinks=false,
  pdfborder={0 0 0},
  naturalnames=false,
  hypertexnames=false,
  breaklinks
}

\usepackage[capitalize]{cleveref}
\crefname{section}{section}{sections}
\crefname{subsection}{subsection}{subsections}

\Crefname{figure}{Figure}{Figures}

\crefformat{equation}{\textup{#2(#1)#3}}
\crefrangeformat{equation}{\textup{#3(#1)#4--#5(#2)#6}}
\crefmultiformat{equation}{\textup{#2(#1)#3}}{--\textup{#2(#1)#3}}
{, \textup{#2(#1)#3}}{, and \textup{#2(#1)#3}}
\crefrangemultiformat{equation}{\textup{#3(#1)#4--#5(#2)#6}}%
{ and \textup{#3(#1)#4--#5(#2)#6}}{, \textup{#3(#1)#4--#5(#2)#6}}{, and \textup{#3(#1)#4--#5(#2)#6}}

\Crefformat{equation}{#2Equation~\textup{(#1)}#3}
\Crefrangeformat{equation}{Equations~\textup{#3(#1)#4--#5(#2)#6}}
\Crefmultiformat{equation}{Equations~\textup{#2(#1)#3}}{ and \textup{#2(#1)#3}}
{, \textup{#2(#1)#3}}{, and \textup{#2(#1)#3}}
\Crefrangemultiformat{equation}{Equations~\textup{#3(#1)#4--#5(#2)#6}}%
{ and \textup{#3(#1)#4--#5(#2)#6}}{, \textup{#3(#1)#4--#5(#2)#6}}{, and \textup{#3(#1)#4--#5(#2)#6}}

\crefdefaultlabelformat{#2\textup{#1}#3}

\definecolor{darkgreen}{rgb}{0,0.5,0}
\definecolor{burgundy}{rgb}{0.5, 0.0, 0.13}

\newtheorem{theorem}{Theorem}
\newtheorem{lemma}[theorem]{Lemma}
\newtheorem{definition}[theorem]{Definition}

\newtheorem{remark}[theorem]{Remark}

\def\R{\mathbb{R}} 
\def\set#1#2{\left\{#1\,\big|\,#2\right\}}
\newcommand{\norm}[3][]{#1\|#2#1\|_{#3}}
\newcommand{\enorm}[3][]{#1|\hspace*{-.6mm}#1|\hspace*{-.6mm}#1|#2#1|
                         \hspace*{-.6mm}#1|\hspace*{-.6mm}#1|_{#3}}
\def\Lop#1#2{L\big(#1;#2\big)}

\newcommand{\spl}[3]{(#1,#2)_{#3}}
\newcommand{\spe}[3]{\langle#1,#2\rangle_{#3}}

\newcommand{\jump}[2][]{%
 \ifthenelse{\equal{#1}{}}
  {[\hspace*{-1.8mm}[\,#2\,]\hspace*{-1.8mm}]}%
 {\left[\hspace*{-1.2mm}\left[#2\right]\hspace*{-1.2mm}\right]}%
}
\def\I{\mathbf{I}} 

\def\A{\mathbf{A}} 
\def\b{\mathbf{b}}
\def\r{c}
\def\u{\mathbf{u}} 
\def\v{\mathbf{v}} %

\def\TT{\mathcal{T}} 
\def\DD{\mathcal{D}}  
\def\NN{\mathcal{N}} 

\def\EE{\mathcal{E}} 
\def\EEr{\EE_{\Gamma}} 
\def\EEi{\EE_I} 
\def\EEt{\EE_T} 

\def\AA{\mathcal{A}}
\def\BB{\mathcal{B}}
\def\HH{\mathcal{H}}
\def\OO{\mathcal{O}}

\def\normal{\mathbf{n}}

\def\CC{\mathcal{C}} 
\def\PP{\mathcal{P}} 

\def\VV{\mathcal{V}} 
\def\KK{\mathcal{K}} 
\def\SS{\mathcal{S}} 

\def\II{\mathcal{I}} 
\def\IIh{\II_h} 

\def\diam{{\operatorname{diam}}}
\def\div{\operatorname{div}}
\def\loc{{\operatorname{\ell oc}}}

\title[Adaptive non-symmetric FVM-BEM]
      {An adaptive non-symmetric finite volume and boundary element coupling method for a
       fluid mechanics interface problem}

\author{Christoph Erath}
\address{TU Darmstadt, Department of Mathematics, Dolivostra\ss{}e 15, 64293 Darmstadt, Germany}
\email{erath@mathematik.tu-darmstadt.de}

\thanks{C. Erath (corresponding author): TU Darmstadt, Germany; erath@mathematik.tu-darmstadt.de}
\thanks{R. Schorr: TU Darmstadt, Germany; schorr@gsc.tu-darmstadt.de}

\author{Robert Schorr}
\address{TU Darmstadt, Graduate School of Computational Engineering, Dolivostra\ss{}e 15, 64293 Darmstadt, Germany}
\email{schorr@gsc.tu-darmstadt.de}

\date{\bf\today}
\begin{document}

\begin{abstract}
We consider an interface problem often arising in transport problems: a coupled 
system of partial differential equations with one (elliptic) transport equation on a 
bounded domain and one equation (in this case the Laplace problem) on the complement, an 
unbounded domain. 
Based on the non-symmetric coupling of the finite volume method and boundary element
method of~\cite{Erath:2015-2}
we introduce a robust residual error estimator. 
The upper bound of the error in an energy (semi)norm 
is robust against variation of the model data. 
The lower bound, however, additionally
depends on the P\'eclet number.
In several examples we use the local contributions of the a~posteriori error estimator 
to steer an adaptive mesh-refining algorithm.
The adaptive FVM-BEM coupling turns out to be an efficient method especially to 
solve problems from fluid mechanics, mainly because of the local flux 
conservation and the stable approximation of convection dominated problems.\\[0.5\baselineskip]


\end{abstract}

\maketitle
\section{Introduction and model problem}

We consider the prototype for flow and transport in porous media in an
interior domain and a homogeneous diffusion process in the corresponding unbounded exterior
problem. To approximate such problems the coupling of the finite volume method (FVM) 
and the boundary element method (BEM) is of particular interest.
For the 
vertex-centered FVM-BEM we refer to~\cite{Erath:2012-1} and
for the cell-centered FVM-BEM to~\cite{Erath:2013-2}.
Note that the coupling of FVM and BEM conserves mass, provides a stable approximation
also for convection dominated problems (option of an upwind stabilization) in the interior domain, 
and avoids the truncation of the unbounded exterior domain due to a transformation 
of the exterior problem into
an integral equation.
We can also interpret the model that the (unbounded) exterior problem ``replaces'' 
the (unknown) boundary conditions of the interior problem~\cite[Remark 2.1]{Erath:2012-1}. 
Recently, the non-symmetric vertex-centered FVM-BEM coupling approach
was introduced in~\cite{Erath:2015-2}, which results in a smaller system
of linear equations than the previous three field coupling approach cited above. 
However, a~posteriori estimators for this kind of 
FVM-BEM coupling were not developed.
Note that for uniform mesh refinement, optimal convergence order can only be guaranteed
if the solution has enough regularity~\cite{Erath:2015-2}, which is usually not met in practice.
Computable local contributions of a~posteriori estimators can be used
to refine a mesh for a numerical scheme, where the error appears to be large and thus 
might lead to an improved convergence rate.

In general, a~posteriori estimators bound the error from above (reliability) and 
below (efficiency). Probably the most widespread a~posteriori estimates are of residual type; 
see, e.g.,~\cite{Verfurth:book-1996} for a survey in the context of finite element methods (FEM)
for the Poisson problem.
Estimators for FEM-BEM couplings are also well-established. In~\cite{Aurada:2013-1}
the authors provide a good overview of residual-based a~posteriori
estimates for different FEM-BEM coupling strategies, also for a non-symmetric FEM-BEM
coupling, but only for a diffusion operator. Since we consider a convection diffusion reaction
problem, we have a special focus on robust estimates, i.e., estimates which should not depend on 
the variation of the diffusion, dominated convection and reaction.
Additionally, we do not assume a strong coerciveness assumption for the convection reaction terms. 
Note that the estimates have to be done in a certain energy (semi)norm.
Therefore, the ellipticity estimate for a stabilized bilinear form of the problem 
from~\cite[Theorem 4]{Erath:2015-2}
(or~\cite{Of:2013-1,Aurada:2013-1} for pure diffusion problems) is not directly 
applicable since the dependency of the constant can not be stated explicitly 
for an estimate in the energy (semi)norm. Thus, we prove an ellipticity estimate in 
the energy (semi)norm for the
original bilinear form in \cref{lem:robustellipt}. This estimate is only valid
if the minimal eigenvalue is bigger than a computable bound, which
depends on an arbitrary but fixed $\varepsilon\in (0,1)$ and the contraction constant of the double
layer integral operator. Similar to the discussion 
in~\cite{Of:2013-1,Aurada:2013-1,Erath:2015-2} this seems to be a theoretical restriction. 
Finally, our constant of the ellipticity estimate depends on the minimal eigenvalue
of the diffusion matrix and $\varepsilon$. However, if we know the minimal eigenvalue
we can calculate the constant explicitly. Hence, in the following we consider this
estimate as robust having chosen the diffusion big enough; see \cref{rem:robreliability}. 
The proof of reliability relies on a robust interpolation operator known from
the finite element literature~\cite{Petzoldt:2002-1}. Note that the diffusion distribution
has to be quasi-monotone over a primal triangulation. Thus, to simplify notation,
we present the robust estimator only for piecewise constant diffusion.
Contrary to the analysis in~\cite{Erath:2013-1} the proof starts with the 
robust ellipticity estimate. Since our system does not provide a ``global'' 
Galerkin orthogonality
(in contrast to a classical FEM-BEM coupling)
we use an $L^2$-orthogonality property of the residual to integrate a piecewise constant
approximation of the error and add and substract the robust interpolation of the error. 
This allows us to use some robust estimates of residual and jump terms; see~\cite{Erath:2013-1}.
Furthermore, the Galerkin orthogonality of the BEM part and some standard localizations
complete the proof and show \cref{th:reliabilityfvem}. 
Note that the fully computable, robust local refinement indicators 
consist of a residual and normal jump quantities (including jump terms on the coupling boundary) 
with factors, which ensure robustness.
A tangential jump measures the error in tangential direction on the coupling boundary.
The upwind stabilization adds an additional quantity to our estimator, which measures
the amount of upwinding. 
To complete the theory we also state a non-robust version of an estimator in \cref{th:reliabilitynonrobust}. 
There, we directly use the ellipticity estimate of~\cite[Theorem~4]{Erath:2015-2}
for a stabilized bilinear form. 
As in~\cite{Aurada:2013-1} for a pure diffusion operator we show that
this stabilized bilinear form evaluated for the errors is equal to the standard bilinear form.
The rest of the proof is standard and follows the lines above using
non-robust techniques such as the classical Cl\'ement nodal interpolant~\cite{Clement:1975-1}.
We remark that in this case the quasi-monotonicity of the diffusion is not necessary.

The efficiency follows mostly from~\cite{Erath:2013-1} and is stated in \cref{th:efficiency}. 
Therefore, we only present
the local estimate from a contribution which differs.
In summary, the estimator is local and, in case of a quasi-uniform mesh on the
boundary, also generically efficient. We stress that even for the FEM-BEM coupling there is
no better result available in the literature~\cite{Aurada:2013-1}. 
However, the lower bound is indeed robust against
discontinuities of the diffusion coefficient and a dominating reaction term but
still depends on the local P\'eclet number for convection problems.
This property is typical for estimates in the energy norm.
To get fully robustness one would have to introduce additionally an augmented norm, which
absorbs the convection terms. We note, however, that this norm is not computable and 
we could not prove an upper bound for this extended norm because we do not have homogeneous
Dirichlet boundary conditions.
For more details we refer to~\cite[Remark 6.1.]{Erath:2013-1}.

Throughout, we denote by $L^m(\cdot)$ and $H^m(\cdot)$, $m>0$, the standard Lebesgue 
and Sobolev spaces equipped with the corresponding norms $\norm{\cdot}{L^2(\cdot)}$ and
$\norm{\cdot}{H^m(\cdot)}$. 
Our domain $\Omega\subset \R^d$, $d=2,3$ will be a bounded domain with 
connected polygonal Lipschitz boundary $\Gamma$.
For $\omega\subset\Omega$, $\spl{\cdot}{\cdot}{\omega}$ is the $L^2$ scalar product. The space $H^{m-1/2}(\Gamma)$
is the space of all traces of functions from $H^m(\Omega)$ and the duality
between $H^{m}(\Gamma)$ and $H^{-m}(\Gamma)$ is given by the extended 
$L^2$-scalar product $\spe{\cdot}{\cdot}{\Gamma}$. 
In $H^1_{\loc}(\Omega):=\set{v:\Omega\to\R}{v|_K\in H^1(K), \, \text{for all } K\subset \Omega \,\mbox{open and bounded}}$ 
we collect all 
functions with local $H^1$ behavior. Furthermore, the Sobolev space $W^{1,\infty}$ contains exactly the Lipschitz 
continuous functions. If it is clear from the context, we do not use a notational 
difference for functions in a domain and their traces.
To simplify the presentation we equip the 
space $\HH:=H^1(\Omega)\times H^{-1/2}(\Gamma)$
with the norm
\begin{align*}
	\norm{\v}{\HH}^2:=\norm{v}{H^1(\Omega)}^2+\norm{\psi}{H^{-1/2}(\Gamma)}^2
\end{align*}
for $\v=(v,\psi)\in\HH$.

This allows us to specify our model problem in detail.
Let $\Omega$  be defined as above
and let $\Omega_e=\R^d\backslash\overline{\Omega}$ be the corresponding
unbounded exterior domain. 
The coupling boundary 
$\Gamma=\partial \Omega= \partial\Omega_e$ is divided in an inflow and outflow part, 
namely $\Gamma^{in}:=\set{x\in\Gamma}{\b(x)\cdot\normal(x)<0}$ and
$\Gamma^{out}:=\set{x\in\Gamma}{\b(x)\cdot\normal(x)\geq 0}$, respectively, 
where $\normal$ is the normal vector on $\Gamma$ pointing outward with respect
to $\Omega$.

We consider the same model problem as in~\cite{Erath:2012-1,Erath:2013-2, Erath:2015-2}
which reads in a weak sense:
find $u\in H^1(\Omega)$ and 
$u_e \in H^1_{\loc}(\Omega_e)$ such that
\begin{subequations}
\label{eq:model}
\begin{alignat}{2}
\label{eq1:model}
   \div (-\A \nabla u + \b u)+\r u  &= f \quad & &\text{in }\Omega,\\
\label{eq2:model}
-\Delta u_e &=0 \quad & &\text{in } \Omega_e,\\
\label{eq3:model}
u_e(x)&=C_{\infty}\log|x|+\OO(1/|x|) \quad& &\text{for }|x|\to \infty,\quad d=2,\\
\label{eq3a:model}
u_e(x)&=\OO(1/|x|) \quad& &\text{for }|x|\to \infty,\quad d=3,\\
\label{eq4:model}
 u&=u_e+u_0 \quad & &\text{on } \Gamma, \\
 \label{eq5:model}
 ( \A \nabla u-\b u)\cdot\normal&=
 \frac{\partial u_e}{\partial \normal}+t_0\quad & &\text{on }  
 \Gamma^{in},\\
 \label{eq6:model}
 ( \A \nabla u)\cdot\normal&=
 \frac{\partial u_e}{\partial \normal}+t_0\quad & &\text{on } \Gamma^{out}.
\end{alignat}
\end{subequations}
The diffusion matrix $\A:\Omega\to\R^{d\times d}$ has piecewise Lipschitz continuous entries; i.e., entries in 
$W^{1,\infty}(T)$ for every $T\in\TT$, where $\TT$ is a mesh of $\Omega$ introduced below in 
\cref{subsec:triangulation}. 
Additionally, $\A$ is bounded, symmetric and uniformly positive definite, i.e.,
there exist positive constants $C_{\A,1}$ and $C_{\A,2}$
with
$C_{\A,1}|\mathbf{v}|^2\leq \mathbf{v}^T\A(x)\mathbf{v}
	        \leq C_{\A,2}|\mathbf{v}|^2$
for all $\mathbf{v}\in \R^d$ and almost every $x\in\Omega$.
The best constant $C_{\A,1}$ equals the infimum over $x\in \Omega$ of the minimum eigenvalue of $\A(x)$, 
which we will denote $\lambda_{\min}(\A)$.
Note that this includes coefficients $\A$ that are $\TT$-piecewise constant.
Furthermore, $\b\in W^{1,\infty}(\Omega)^d$
and $\r\in L^{\infty}(\Omega)$
satisfy the weak coerciveness assumption 
\begin{align}
  \label{eq:bcestimate1}
  \frac{1}{2}\div\b(x)+\r(x)\geq 0\quad \text{for almost every }x\in\Omega.
\end{align}
We stress that our analysis holds
for constant $\b$ and $\r=0$ as well.
Finally, we choose the right-hand side $f\in L^2(\Omega)$, 
and allow prescribed jumps $u_0\in H^{1/2}(\Gamma)$,
and $t_0\in H^{-1/2}(\Gamma)$. In the two dimensional
case we additionally assume $\diam(\Omega)<1$ which can always be achieved 
by scaling to ensure $H^{-1/2}(\Gamma)$ ellipticity 
of the single layer operator defined below.
The constant $C_{\infty}$ is unknown; see~\cite{McLean:2000-book,Erath:2012-1,Erath:2015-2} for possible
different radiation conditions.
The model problem~\cref{eq:model} admits a unique solution for both, the 
two and three dimensional case; see~\cite{Erath:2012-1}.

The content of this paper is organized as follows. 
\Cref{sec:integralweak} gives a short summary on integral equations and
the weak formulation of our model problem based on the non-symmetric coupling
approach. 
In \cref{sec:FVMBEM} we introduce the non-symmetric FVM-BEM coupling to solve our 
model problem.
\Cref{sec:estimator} introduces a robust a~posteriori error estimator and shows reliability
and efficiency.
Numerical experiments, found in \cref{sec:numerics},
confirm the theoretical findings. Some conclusions complete the work.
%
\section{Integral equation and weak coupling formulation}
\label{sec:integralweak}

We consider a weak form of the model problem~\cref{eq:model} in terms of boundary integral
operators~\cite{Erath:2015-2}. Then  
the coupling reads:
find 
$u\in H^1(\Omega)$, $\phi\in H^{-1/2}(\Gamma)$
such that 
\begin{subequations}
\begin{align}
\label{eq1:weakfem}
   \AA(u,v)-\spe{\phi}{v}{\Gamma}
   &= \spl{f}{v}{\Omega}+\spe{t_0}{v}{\Gamma},\\
   \label{eq2:weakfem} 
   \spe{\psi}{(1/2-\KK)u}{\Gamma}+\spe{\psi}{\VV\phi}{\Gamma}&=\spe{\psi}{(1/2-\KK)u_0}{\Gamma}
\end{align}
\end{subequations}
for all 
$v\in H^1(\Omega)$, $\psi\in H^{-1/2}(\Gamma)$ with the bilinear form 
\begin{align*}
  \AA(u,v):=\spl{\A\nabla u-\b u}{\nabla v}{\Omega}+\spl{\r u}{v}{\Omega}
+\spe{\b\cdot\normal \,u}{v}{\Gamma^{out}}.
\end{align*}
The single layer operator $\VV$ and the double layer
operator $\KK$ are given, for smooth enough input, by
\begin{align*}
(\VV \psi)(x)=\int_{\Gamma}\psi(y)G(x-y)\,ds_{y}
 	\qquad
(\KK \theta)(x)=
  \int_{\Gamma}\theta(y)\frac{\partial}{\partial 
  \normal_{y}}G(x-y)\,ds_{y} \qquad x\in \Gamma,
\end{align*}
where $\normal_y$ is a normal vector with respect to $y$ and
$G(z)=-\frac{1}{2\pi}\log|z|$ for the 2-D case
and $G(z)=\frac{1}{4\pi}\frac{1}{|z|}$ for the 3-D case is the fundamental solution for the Laplace operator.
We recall~\cite[Theorem 1]{Costabel:1988-1} that  these operators can be extended to bounded operators
\begin{align*}
\VV \in\Lop{H^{s-1/2}(\Gamma)}{H^{s+1/2}(\Gamma)},
	\qquad
\KK \in\Lop{H^{s+1/2}(\Gamma)}{H^{s+1/2}(\Gamma)},
	\quad s\in [-\tfrac12,\tfrac12].  
\end{align*}
It is also well-known that $\VV$ is symmetric and $H^{-1/2}(\Gamma)$ elliptic. The expression
\begin{align*}
\norm{\cdot}{\VV}^2:=\spe{\VV\cdot}{\cdot}{\Gamma}
\end{align*}
defines a norm in $H^{-1/2}(\Gamma)$. This norm is equivalent to $\norm{\cdot}{H^{-1/2}(\Gamma)}$.
In this work, we will also use the 
contraction constant $C_{\KK}\in[1/2,1)$ from~\cite{Steinbach:2001-1} for the double layer operator $\KK$.

%

%
For convenience the system~\cref{eq1:weakfem,eq2:weakfem} can be written in the product space $\HH=H^1(\Omega)\times H^{-1/2}(\Gamma)$ as follows: we introduce the bilinear form $\BB:\HH\times\HH\to \R$
\begin{align}
	\label{eq:Bfem}
	\BB((u,\phi);(v,\psi))&:=   \AA(u,v)-\spe{\phi}{v}{\Gamma}
      +\spe{\psi}{(1/2-\KK)u}{\Gamma}+\spe{\psi}{\VV\phi}{\Gamma},
\end{align}
and the linear functional
\begin{align}
	\label{eq:Ffem}
	F((v,\psi)):=  \spl{f}{v}{\Omega} +\spe{t_0}{v}{\Gamma}+\spe{\psi}{(1/2-\KK)u_0}{\Gamma}.  
\end{align}
Then~\cref{eq1:weakfem,eq2:weakfem} is equivalent to: find
$\u\in\HH$ such that
\begin{align}
	\label{eq:FEMBEM}
	\BB(\u;\v)=F(\v) \qquad \text{for all } \v\in \HH.
\end{align}

\section{A non-symmetric FVM-BEM coupling}
\label{sec:FVMBEM}
In this section we shortly present the non-symmetric FVM-BEM coupling discretization introduced in \cite{Erath:2015-2}. 
From now on we assume $t_0\in L^2(\Gamma)$.
First, let us introduce the notation for the triangulation and some discrete function
spaces.
\subsection{Triangulation}
\label{subsec:triangulation}
Throughout, $\TT$ denotes a triangulation or primal mesh of $\Omega$,
$\NN$ and $\EE$ are the corresponding set of nodes and edges/faces, respectively.
The elements $T\in\TT$ are non-degenerate triangles (2-D case) or tetrahedra (3-D case),
and considered to be closed.
For the Euclidean diameter of $T\in\TT$ we write 
$h_T:=\sup_{x,y\in T}|x-y|$.
Moreover, $h_E$ denotes the length of an edge or Euclidean diameter of $E\in\EE$.
The triangulation is regular
in the sense of Ciarlet~\cite{Ciarlet:1978-book},
i.e., the ratio of the diameter $h_T$ of any element $T\in\TT$ to the diameter
of its largest inscribed ball is bounded by a constant independent
of $h_T$, the so called shape-regularity constant. 
Additionally, we assume that the triangulation $\TT$ is aligned 
with the discontinuities
of the coefficients $\A$, $\b$, and $\r$ of the differential equation (if any) and of the data
$f$, $u_0$, and $t_0$. Throughout, if $\normal$ appears in a boundary
integral, it denotes the unit normal vector to the boundary pointing outward the
domain.
We denote by $\EEt\subset\EE$ the set of all edges/faces of $T$, i.e.,
$\EEt:=\set{E\in\EE}{E\subset \partial T}$ and by
$\EEr:=\set{E\in\EE}{E\subset\Gamma}$ the set of 
all edges/faces on the boundary $\Gamma$.
%
%
%
%
%
\begin{figure}[t]
  \centering
  \psfrag{v1}[c][c]{\small $V_{1}$}\psfrag{v2}[c][c]{\small $V_{2}$}
  \psfrag{v3}[c][c]{\small $V_{3}$}
  \psfrag{v4}[c][c]{\small $V_{4}$}\psfrag{v5}[c][c]{\small $V_{5}$}
  \psfrag{v6}[c][c]{\small $V_{6}$}
  \psfrag{v7}[c][c]{\small $V_{7}$}
  \psfrag{a1}[c][c]{\small $a_{1}$}\psfrag{a2}[c][c]{\small $a_{2}$}
  \psfrag{a3}[c][c]{\small $a_{3}$}
  \psfrag{a4}[c][c]{\small $a_{4}$}\psfrag{a5}[c][c]{\small $a_{5}$}
  \psfrag{a6}[c][c]{\small $a_{6}$}
  \psfrag{a7}[c][c]{\small $a_{7}$}  
  \subfigure[Constructions of $\TT^*$.]
  {\label{fig:dualconst}\includegraphics[height=0.4\textwidth]{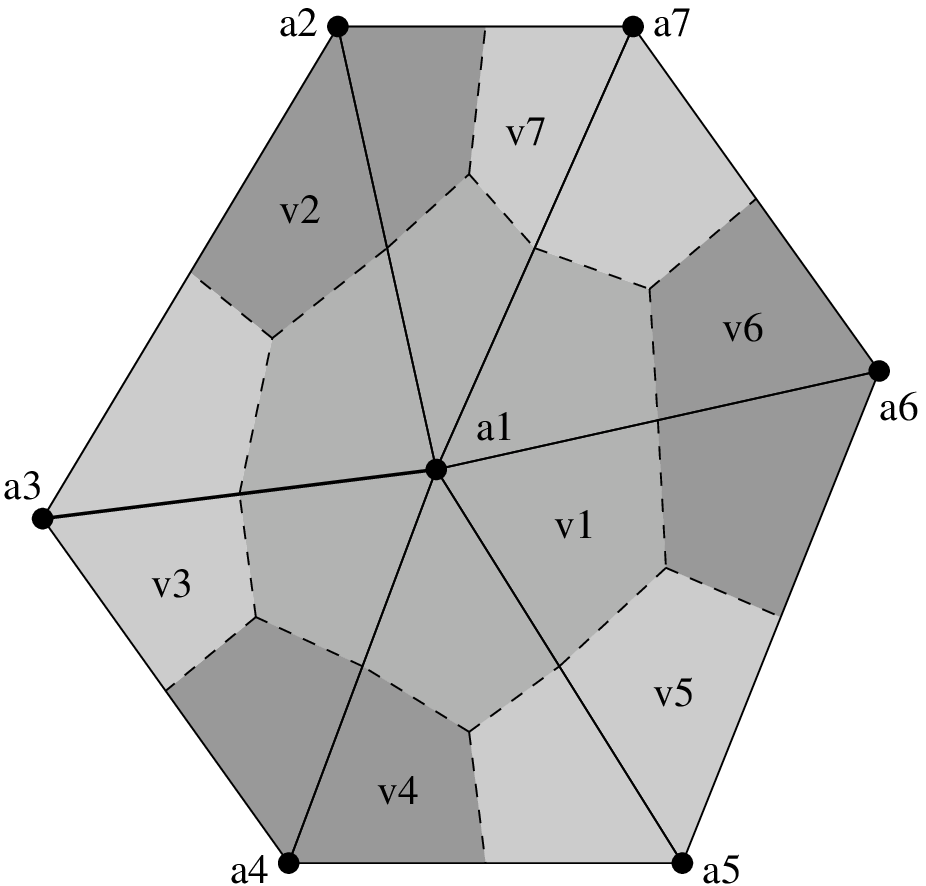}}
  \hspace*{0.1\textwidth}
  \psfrag{t17}[c][c]{\small $\tau_{17}$}
  \psfrag{t34}[c][c]{\small $\tau_{34}$}
  \subfigure[Edges for upwinding.]
  {\label{fig:edgeupwind}\includegraphics[height=0.4\textwidth]{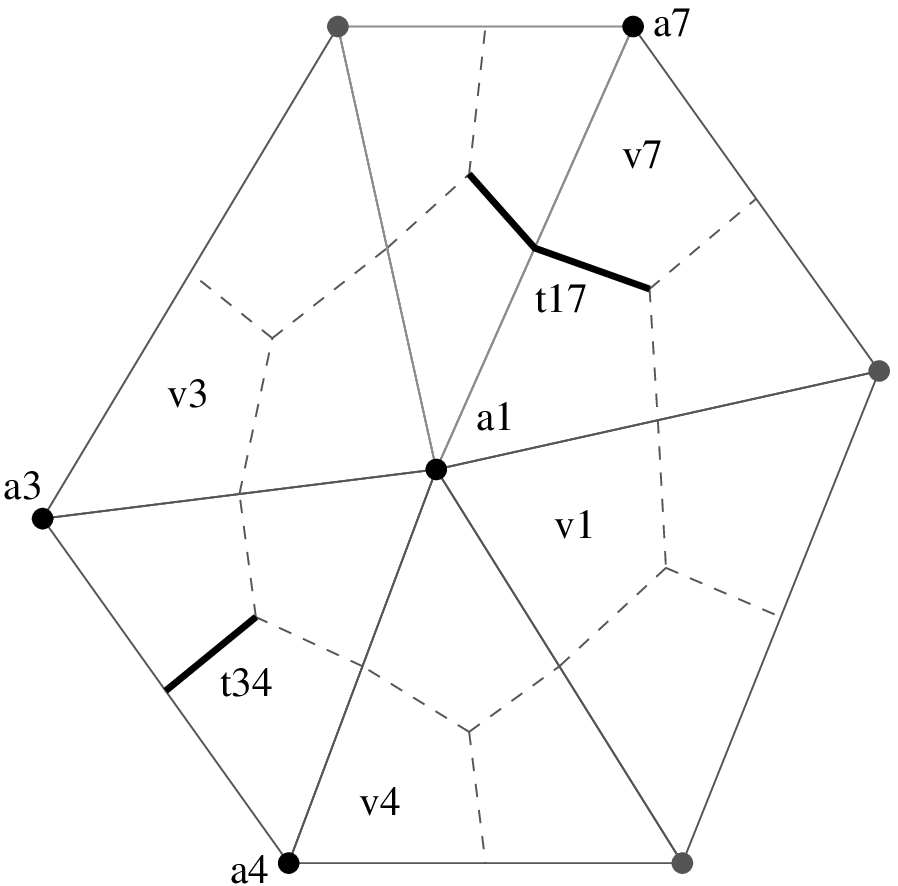}}
  \caption{
  The construction of the dual mesh $\TT^*$ from the
  primal mesh $\TT$ in two dimensions with the center of gravity point in
  the interior of the elements in Figure~\subref{fig:dualconst}; 
  the dashed lines (gray boxes) are the new control volumes $V_i$ of $\TT^*$ and are 
  associated with $a_i\in\NN$.
  In Figure~\subref{fig:edgeupwind} we see an example
  intersection $\tau_{17}=V_1\cap V_7\neq\emptyset$ of two neighboring cells
  $V_1,V_7 \in \TT^*$, where $\tau_{17}$ is the union of 
  two straight segments. For $a_3, a_4\in \NN$, where both $a_3$ and $a_4$ 
  lie on $\Gamma$, $\tau_{34}=V_3\cap V_4\neq\emptyset$ 
  is only a single segment.\label{fig:primaldual}}
\end{figure}
\subsection*{Dual mesh}
We construct the dual mesh $\TT^*$ from the primal mesh $\TT$ as follows.
In two dimensions we connect the center of gravity of an element $T\in\TT$ 
with the midpoint of the edges $E\in\EEt$; see 
\cref{fig:dualconst}, where 
the dashed lines are the new boxes, called control volumes. 
In three dimensions we connect the center of gravity of an element
$T\in\TT$ with the centers of gravity of the four faces $E\in\EEt$.
Furthermore, each center of gravity of a face $E\in\EEt$ is connected by straight lines
to the midpoints of its edges. The elements of this dual mesh $\TT^*$ are taken to be closed. Note that they are non-degenerate domains because of the non-degeneracy of the elements of the primal mesh.
Given a vertex $a_i\in\NN$ from the primal mesh $\TT$ ($i=1\ldots \#\NN$), there exists a unique box containing $a_i$. We thus number the elements of the dual mesh $V_i\in \TT^*$, following the numbering of vertices.
%
%
%
%
\subsection*{Discrete function spaces}
We define by
$\SS^1(\TT):=\set{v\in\CC(\Omega)}{v|_T \text{ affine for all } T\in\TT}$
the piecewise affine and globally continuous function space 
on $\TT$.
The space $\PP^0(\EEr)$ is the $\EEr$-piecewise constant function space. 
On the dual mesh $\TT^*$ we provide
$\PP^0(\TT^*):=\set{v\in L^2(\Omega)}{v|_V \text{ constant } V\in\TT^*}$.
With the aid of the characteristic function $\chi_i^*$ over the volume $V_i$ 
we can write $v_h^*\in\PP^0(\TT^*)$ as
\begin{align*}
  v_h^*=\sum_{x_i\in\NN}v_i^*\chi_i^*,
\end{align*}
with real coefficients $v_i^*$.
Furthermore, we define the $\TT^*$-piecewise constant interpolation operator
\begin{align}
  \label{eq:intoppiecewise}
  \IIh^*:\CC(\overline\Omega)\to\PP^0(\TT^*),\quad
  \IIh^*v:=\sum_{a_i\in\NN}v(a_i)\chi_i^*(x).
\end{align}

%
\subsection{The discrete system}
With these preparations made we can introduce the non-symmetric FVM-BEM coupling method, 
which reads:
find $u_h\in\SS^1(\TT)$ and $\phi_h\in\PP^0(\EEr)$ such that
\begin{subequations}
 \label{eq:fvm}
\begin{align}
\label{eq1:fvm}
  \AA_V(u_h,v_h)-\spe{\phi_h}{\IIh^*v_h}{\Gamma}&=\spl{f}{\IIh^*v_h}{\Omega}
  +\spe{t_0}{\IIh^*v_h}{\Gamma},\\
  \label{eq2:fvm}   
  \spe{\psi_h}{(1/2-\KK)u_h}{\Gamma}+\spe{\psi_h}{\VV\phi_h}{\Gamma}
  &=\spe{\psi_h}{(1/2-\KK)u_0}{\Gamma}	
\end{align}
\end{subequations}
for all $v_h\in\SS^1(\TT)$, $\psi_h\in\PP^0(\EEr)$
with the finite volume bilinear form $\AA_V:\SS^1(\TT)\times \SS^1(\TT)\to \mathbb R$ given by
\begin{align}
  \label{eq:fvmbilinear}
  \begin{split}
  \AA_V(u_h,v_h)&:=\sum_{a_i\in\NN}v_h(a_i)\bigg(
  \int_{\partial V_i\backslash\Gamma}(-\A \nabla u_h+\b u_h ) \cdot \normal\,ds\\
  &\qquad\qquad\quad+\int_{V_i}\r u_h\,dx 
  +\int_{\partial V_i\cap\Gamma^{out}}
	\b\cdot\normal\,u_h\,ds\bigg).
	\end{split}
\end{align}

A more detailed derivation can be found in~\cite{Erath:2015-2}.
 
\begin{remark}
Note that the trial and test spaces are in fact different.
\end{remark}
It is well known that the FVM with the central approximation of the convention term
leads to strong instabilities for convection dominated problems.
Finite volume schemes, however, allow an
easy upwind stabilization; see~\cite{Roos:1996-book}. 
Although there exist several upwinding possibilities, we focus 
on the classical full upwinding in this work.

If we want to apply an upwind scheme for the finite volume scheme,
we replace $\b u_h$ 
on the interior dual edges/faces $V_i\backslash \Gamma$
in $\AA_V$~\cref{eq:fvmbilinear} by an upwind approximation. Given $V_i\in \TT^*$, 
we consider the intersections with the neighboring boxes 
$\tau_{ij}=V_i\cap V_j\neq\emptyset$ for $V_j \in \TT^*$. 
Note that in two dimensions $\tau_{ij}$ is the union of two straight segments or 
(when the associated vertices $a_i, a_j\in \NN$ lie on $\Gamma$) a single segment; 
see \cref{fig:edgeupwind}. 
In three dimensions $\tau_{ij}$ consists of one or two polygonal surfaces.
We then compute the average
\begin{align*}
\beta_{ij}:=\frac1{|\tau_{ij}|}\int_{\tau_{ij}}\b\cdot\normal_i\,ds,
\end{align*}
where $\normal_i$ points outwards with respect to $V_i$. 
Then, 
the upwind value $u_{h,ij}$ defined by the classical (full) upwind scheme is

\begin{align}
  \label{eq:upwind}
  u_{h,ij}:=\begin{cases}
   \displaystyle
     u_h(a_j)\quad &\text{for } \beta_{ij}<0,\\[1mm]
   \displaystyle u_h(a_i) \quad &\text{for } \beta_{ij}\geq 0.
  \end{cases}
\end{align}

The analysis in this work also holds
for a weighted upwinding strategy which is used to reduce the 
excessive numerical diffusion; see~\cite{Roos:1996-book,Erath:2015-2}.

Whenever we apply an upwind scheme for the convection part, we
replace the finite volume bilinear form $\AA_V$ in~\cref{eq1:fvm} by
\begin{align}
  \label{eq:fvmbilinearup}
  \begin{split}
  \AA_V^{up}(u_h,v_h)&:=\sum_{a_i\in\NN}v_h(a_i)\bigg(
  \int_{\partial V_i\backslash\Gamma}-\A \nabla u_h \cdot \normal\,ds
  +\int_{V_i}\r u_h\,dx \\
  &\qquad\qquad\quad
  +\sum_{j\in\NN_i}\int_{\tau_{ij}}\b \cdot\normal \,u_{h,ij}\,ds
  +\int_{\partial V_i\cap\Gamma^{out}}
	\b\cdot\normal\,u_h\,ds\bigg).
	\end{split}
\end{align}
where $\NN_i$ denotes the index set of nodes in $\TT$ 
of all neighbors of $a_i \in\NN$.

\section{Residual based a~posteriori error estimator}
\label{sec:estimator}
In this section we will introduce an elementwise refinement indicator on which
our a~posteriori error estimator is based.
In order to do that we define the residual 
\begin{align}
  \label{eq:residuum}
  R:=R(u_h)=f-\div(-\A \nabla u_h +\b u_h)-\r u_h
  \quad\text{on } T\in\TT
\end{align}
and an edge/face-residual or jump $J:L^2(\EE)\to \R$ by
\begin{align}
  \label{eq:edgeJ}
  J|_E:=J(u_h)|_E=
  \begin{cases}\displaystyle
     \big[(-\A \nabla u_h)|_{E,T}-(-\A \nabla u_h)|_{E,T'}\big]\cdot\normal 
     & \begin{array}{ll} \text{for all }E\in\EEi\text{ with }\\ E=T\cap T', T,T'\in\TT \end{array}\\[1mm]
    \displaystyle (-\A\nabla u_h+\b u_h)\cdot\normal+\phi_h+t_0 & 
    \text{ for all }E\in\EEr^{in},\\[1mm]
  \displaystyle   -\A\nabla u_h\cdot\normal+\phi_h+t_0 & 
  \text{ for all }E\in\EEr^{out}.
  \end{cases}
\end{align}

Note that $\varphi_{E,T}$ denotes the trace of $\varphi\in H^1(T)$ on $E$
and the normal vector $\normal$ points from $T$ to $T'$.

\subsection{Robust a posteriori estimation}

For analytical investigations we define the energy (semi)norm
\begin{align}
   \label{eq:energynorm}
   \enorm{v}{\Omega}^2&:=\norm{\A^{1/2}\nabla v}{L^2(\Omega)}^2+
   \norm[\Bigg]{\left(\frac{1}{2}\div\b+\r\right)^{1/2}v}{L^2(\Omega)}^2
   \quad \text{for all }v\in H^1(\Omega).
\end{align}
We stress that there holds with~\cref{eq:bcestimate1} and $\b\cdot\normal \leq 0$ on $\Gamma^{in}$
\begin{align}
\label{eq:bilinearintcoerciv}
  \enorm{v}{\Omega}^2&\leq \AA(v,v).
\end{align}

The following lemma is the key observation for showing a robust upper estimate.

\begin{lemma}\label{lem:robustellipt}
Let us assume $0<\varepsilon<1$ and $(1-\varepsilon)\lambda_{\min}(\A) - \frac{1}{4} C_{\KK} > 0$.
For all $\v=(v,\psi)\in \HH$ there holds
\begin{align}
 \label{eq:robustelipt}
  \BB(\v;\v)  &\geq
  \min\left\{ \varepsilon, C_{\rm harm}\right\}\Big(\enorm{v}{\Omega}^2
  + \norm{\psi}{\VV}^2\Big)
\end{align}	
with the constant
 \begin{align*}  
 C_{\rm harm}=\frac{1}{2}  \left[  
          (1-\varepsilon)\lambda_{\min}(\A) +1  
          - \sqrt{ ((1-\varepsilon)\lambda_{\min}(\A)-1)^2+C_{\KK}}
        \right] 
\end{align*}
and the contraction constant $C_{\KK}\in[1/2,1)$.
\end{lemma} 
\begin{proof}
The proof is similar to the proof in~\cite[Theorem 4]{Erath:2015-2}.
Thus we only sketch the steps that differ.
In the following we denote by $S^{\text{int}}:=\VV^{-1}(1/2+\KK)$ 
the Steklov--Poincar\'e operator, i.e., the Dirichlet to Neumann
map of the interior Laplace problem.
Let $\v=(v,\psi)\in \HH$ be arbitrary.
Thus, \cref{eq:Bfem}, the contractivity property $\spe{\psi}{(1/2+\KK)v}{\Gamma}
  \leq C_{\KK}^{1/2}\spe{S^{\text{int}} v}{v}{\Gamma}^{1/2}\norm{\psi}{\VV}$,
the ellipticity~\cref{eq:bilinearintcoerciv} 
of $\AA(v,v)$ in the (semi)energy norm~\cref{eq:energynorm}, and
the ellipticity of $\VV$
lead to

\begin{align*}
  \BB(\v;\v)  &= 
  \AA(v,v)
  + \spe{\psi}{\VV\psi}{\Gamma}
  - \spe{\psi}{(1/2+\KK)v}{\Gamma}\\  
  &\geq
  \norm{\A^{1/2}\nabla v}{L_2(\Omega)}^2
  +\norm{((\div\b)/2+\r)^{1/2}v}{L_2(\Omega)}^2
  + \norm{\psi}{\VV}^2
   - C_{\KK}^{1/2}\spe{S^{\text{int}} v}{v}{\Gamma}^{1/2}\norm{\psi}{\VV}\\
\end{align*}	
Next, for $0<\varepsilon<1$ we split 
$\norm{\A^{1/2}\nabla v}{L_2(\Omega)}^2=
\varepsilon\norm{\A^{1/2}\nabla v}{L_2(\Omega)}^2+(1-\varepsilon)\norm{\A^{1/2}\nabla v}{L_2(\Omega)}^2$.
With harmonic splitting we build a quadratic form as in~\cite[Theorem 4]{Erath:2015-2}.
Thus, under the assumption that $(1-\varepsilon)\lambda_{\min}(\A) - \frac{1}{4} C_{\KK} > 0$,
$\spe{S^{\text{int}} v}{v}{\Gamma}\geq 0$, 
and with the constant $C_{\rm harm}$
we estimate

\begin{align*}
  \BB(\v;\v)  &\geq
  \varepsilon\norm{\A^{1/2}\nabla v}{L_2(\Omega)}^2
  +\norm{((\div\b)/2+\r)^{1/2}v}{L_2(\Omega)}^2
  + C_{\rm harm}\norm{\psi}{\VV}^2,
\end{align*}	
which proves the assertion.
\end{proof}

Note that~\cref{eq:robustelipt} allows us to prove a robust upper bound. 
However, the diffusion distribution in $\Omega$
has to be quasi-monotone to apply a robust interpolant; 
see also~\cite{Petzoldt:2002-1} in the context of an FEM estimator 
and~\cite{Erath:2013-1} for an FVM-BEM estimator. 
To simplify notation, we restrict ourself to a piecewise constant 
diffusion coefficient $\alpha\in\PP^0(\TT)$ with $\A=\alpha\I$.
Let us suppose that $\Omega$ can be partitioned into a finite number $L$ of open disjoint
subdomains $\Omega_{\ell}$, $1\leq\ell\leq L$ such that the function $\alpha$ is 
equal to a constant $\alpha_{\ell}\in\R$ on each $\Omega_{\ell}$ and the 
triangulation $\TT$ of $\Omega$ fits to $\Omega_{\ell}$; i.e., 
$\partial \Omega_{\ell}$ consists of edges of the underlying triangulation. 
Thus, for two subdomains $\Omega_{k},\Omega_{\ell}$ with 
$\partial\Omega_{k}\cap\partial\Omega_{\ell}\not=\emptyset$ 
we may assume $\alpha_k\not=\alpha_{\ell}$. Otherwise, one can merge $\Omega_k$ and 
$\Omega_{\ell}$ with $\alpha_k=\alpha_{\ell}$ to a new subdomain.

For the $\TT$-piecewise constant function $\alpha\in\PP^0(\TT)$ we write
\begin{align*}
  \alpha_T:=\alpha|_T\quad \text{for all }T\in\TT,
\end{align*}
which obviously gives $\alpha_T=\alpha_{\ell}$ in $\Omega_{\ell}$.

With the definition of the patch of a node $a \in \mathcal{N}$ via
\begin{align*}
 \omega_{a} := \bigcup_{T\in \widetilde{\omega}_{a}} T \quad \text{with }
 \widetilde{\omega}_{a}:= \set{T\in \TT}{a\in \partial T},
\end{align*} 
we can define the set
\begin{align*}
 Q_a:=\bigcup_{T\in\widetilde Q_a}T,
 \quad\text{where}\quad
 \widetilde Q_a :=\set{T \in \widetilde\omega_a}{\alpha_T\geq \alpha_{T'}, 
 \text{ for all }T'\in\widetilde\omega_a}.
\end{align*}
Note that $Q_a$ denotes the union of all simplexes $T\in\widetilde\omega_a$ 
for $a\in\NN$, where $\alpha_T$ achieves a maximum.

\begin{definition}[Quasi-monotonicity~\cite{Petzoldt:2002-1,Erath:2013-1}]
  \label{def:quasimonotone}
  Let $a\in\NN$. We say $\alpha$ is quasi-monotone in $\omega_a$ with respect to
  $a$, if for all elements $T\in\widetilde\omega_a$ there exists a simply 
  connected set $Q_{a,T}$ with $T\cup Q_a\subset Q_{a,T}\subset\omega_a$ such that
  $\alpha_T\leq\alpha_{T'}$ for all $T'\subset Q_{a,T}$, $T'\in \widetilde\omega_a$.
  We call $\alpha$ quasi-monotone, if $\alpha$ is quasi-monotone for all $a\in\NN$.
\end{definition}  

The definitions of~\cite{Petzoldt:2002-1} and~\cite{Erath:2013-1} slightly
differ, since the coupling does not have a Dirichlet boundary. 
This allows us to define a robust nodal interpolant in the sense of~\cite{Petzoldt:2002-1};
\begin{align}
  \label{eq:intopaffine}
  \IIh:H^1(\Omega)\to \SS^1(\TT), \quad \IIh v:=\sum_{a\in\NN} \Pi_a v \,\eta_a,
\end{align}
well known in the context of the finite element method.
Here $\eta_a$ is the standard nodal linear basis function
associated with the node $a$.
The linear and continuous operator $\Pi_a:H^1(\omega)\to \R$ on a domain $\omega\subset\Omega$ 
for a diffusion coefficient with a quasi-monotone
distribution reads
\begin{align*}
 \Pi_a v := \frac{1}{|Q_a|}\int_{Q_a} v\,dx.
\end{align*} 

Before we can introduce a robust refinement indicator, we need some more notation:
First, we define
\begin{alignat*}{2}
\alpha_E&:=\max\big\{\alpha_{T_1},\alpha_{T_2}\big\}  \quad
&&\text{for }E\in\EEi \text{ with } E\subset T_1\cap T_2,\\
 \alpha_E&:=\alpha_T\quad
  &&\text{for }E\in\EEr \text{ with } E\subset \partial T.
\end{alignat*}

%
Besides $\alpha_E$ we define additional quantities; i.e.,
\begin{alignat*}{2}
  \beta_T&:=\min_{x\in T}\Big\{\frac{1}{2}\div\b(x)+\r(x)\Big\}
  \quad
  &&\text{for all } T\in\TT,\\
  \beta_E&:=\min\big\{\beta_{T_1},\beta_{T_2}\big\}\quad
   &&\text{for } E\in\EEi \text{ with }E\subset T_1\cap T_2,\\
  \beta_E&:=\beta_T\quad
   &&\text{for }E\in\EEr \text{ with } E\subset \partial T.
\end{alignat*}
Next, we define  
$\mu_T:=\min\big\{\beta_T^{-1/2},h_T \alpha_T^{-1/2}\big\}$ for all $T\in\TT$ 
and $\mu_E:=\min\big\{\beta_E^{-1/2},h_E\alpha_E^{-1/2}\big\}$ for all $E\in\EE$.
As a notational convention, we take the second argument if $\beta_T=0$ or $\beta_E=0$.

Then, the robust refinement indicator reads for all $T\in\TT$ 
\begin{align}
 \label{eq:refindicatorrobust}
\begin{split}
  \eta_T^2&:=\mu_T^2\norm{R}{L^2(T)}^2+
  \frac{1}{2}\sum_{E\in\EEi\cap\EEt}\alpha_E^{-1/2}\mu_E\norm{J}{L^2(E)}^2
 +\sum_{E\in\EEr\cap\EEt}\alpha_E^{-1/2}\mu_E\norm{J}{L^2(E)}^2\\
  &\quad+\sum_{E\in\EEr\cap\EEt}h_E\norm{
  \nabla_\Gamma\big((1/2-\KK)(u_0-u_h)-\VV\phi_h\big)}{L^2(E)}^2
\end{split}
\end{align}
with $R$ and $J$ from~\cref{eq:residuum} and~\cref{eq:edgeJ},
respectively.
If we apply the upwind discretization~\cref{eq:fvmbilinearup} we additionally need for all $T\in\TT$
\begin{align}
  \label{eq:refindicatorup}
  \eta_{T,{up}}^2&:=\alpha_T^{-1/2}\mu_T
   \sum_{\tau_{ij}^T\in\DD^T}\norm{\b\cdot\normal_i(u_h-u_{h,ij})}{L^2(\tau_{ij}^T)}^2
\end{align}
with $\DD^T:=\set{\tau^T_{ij}}{\tau^T_{ij}=V_i\cap V_j\cap T 
\text{ for }V_i,V_j\in\TT^* \text{ with }V_i\not=V_j, V_i\cap T\not=\emptyset, V_j\cap T\not=\emptyset}$
and $u_{h,ij}$ from \cref{eq:upwind}.
To prove robustness of our a~posteriori estimator we use the following
$L^2$-orthogonality property, which will help us to overcome the lack of Galerkin orthogonality
of the FVM part, and some robust estimates of the residual and jump terms from~\cite{Erath:2013-1};

\begin{lemma}[\cite{Erath:2013-1}]\label{lem:orthorobust}
 Let $\IIh$ be the robust nodal interpolant~\cref{eq:intopaffine} 
 for quasi-monotone diffusion distribution
 in the sense of~\cite{Petzoldt:2002-1} and $\IIh^*$ the $\TT^*$-piecewise constant
 interpolation operator~\cref{eq:intoppiecewise}.
 With the notation above there holds
 \begin{itemize}
 \item for all $v^*\in\PP^0(\TT^*)$ 
 \begin{align}
   \label{eq:proportho}
 \begin{split}
   &\sum_{T\in\TT}\int_T R v^*\,dx
   +\sum_{E\in\EE}\int_{E}J v^*\,ds=0.
 \end{split}
 \end{align}
 \item for all $v\in H^1(\Omega)$, and $v_h=\IIh v\in\SS^1(\TT)$
 \begin{align}
   \label{eq:resfemin}
   \sum_{T\in\TT}\int_T R(v-v_h)\,dx & \lesssim
   \left(\sum_{T\in\TT}\min\left\{\beta_T^{-1},h_T^2 \alpha_T^{-1}\right\}
    \norm{R}{L^2(T)}^2\right)^{1/2}\enorm{v}{\Omega},\\
      \label{eq:edgeJfemin}
      \sum_{E\in\EE}\int_{E}J (v-v_h)\,ds & \lesssim
      \left(\sum_{E\in\EE}\alpha_E^{-1/2}\min\left\{\beta_E^{-1/2},h_E\alpha_E^{-1/2}\right\}
      \norm{J}{L^2(E)}^2\right)^{1/2}\enorm{v}{\Omega}. 
 \end{align}
 \item for all $v\in H^1(\Omega)$, $v_h=\IIh v\in\SS^1(\TT)$, and $v_h^*=\IIh^* v_h\in\PP^0(\TT^*)$ 
 \begin{align}
   \label{eq:resfvmin}
   \sum_{T\in\TT}\int_T R(v_h-v_h^*)\,dx & \lesssim
   \left(\sum_{T\in\TT}\min\left\{\beta_T^{-1},h_T^2 \alpha_T^{-1}\right\}
    \norm{R}{L^2(T)}^2\right)^{1/2}\enorm{v}{\Omega},\\
  \label{eq:edgeJfvmin}
   \sum_{E\in\EE}\int_{E}J (v_h-v_h^*)\,ds & \lesssim
  \left(\sum_{E\in\EE}
  \alpha_E^{-1/2}\min\left\{\beta_E^{-1/2},h_E\alpha_E^{-1/2}\right\}
  \norm{J}{L^2(E)}^2\right)^{1/2}\enorm{v}{\Omega}. 
  \end{align}
  \end{itemize}
\end{lemma} 
 
The next lemma describes the localization of the Sobolev norm on the boundary.
It is well-known in the context of a~posteriori
estimates for boundary element methods; 
e.g.,~\cite[Theorem~1]{Carstensen:1997-1} for the two dimensional 
case and~\cite[Theorem~3.2 and Corollary 4.2]{Carstensen:2001-1} 
for the three dimensional case.
In the following, $\nabla_\Gamma$ denotes the arc length derivative
in the 2-D case or the gradient over the surface in the 3-D case.
\begin{lemma}\label{lem:localization}
Assume $v\in H^{1}(\Gamma)$ is $L^2(\Gamma)$-orthogonal to $\PP^0(\EEr)$. Then, there holds
\begin{align}
\label{eq:p12estimate}
  \norm{v}{H^{1/2}(\Gamma)}\leq C(\EEr) \left(\sum_{E\in\EEr}
  h_E\norm{\nabla_\Gamma v}{L^2(E)}^2\right)^{1/2}.
\end{align}
\end{lemma}

\begin{remark}\rm
The constant $C(\EEr)$ depends on the (boundary-) mesh $\EEr$, but we can 
ensure its boundedness by shape regularity of $\TT$ in two dimensions and 
by only using newest vertex bisection refinement in the 3-D case.
We refer to~{\rm\cite{Carstensen:2001-1}} for a detailed discussion
about the dependency.
\end{remark}


Standard techniques for residual-based error estimates together with
\cref{lem:robustellipt,lem:orthorobust,lem:localization} allow us to show:

\begin{theorem}[Robust reliability]\label{th:reliabilityfvem}
Let us assume $0<\varepsilon<1$ and $(1-\varepsilon)\alpha_{\min} - \frac{1}{4} C_{\KK} > 0$, where
$\alpha_{min}=\min_{T\in\TT} \alpha_T$.
Then, there is a constant $C_{\rm rel}>0$ which depends only on the shape of the
elements in $\TT$ but not on the size, the number of elements nor the variation
of the model data such that
\begin{align}
 \label{eq:robreliability}
  \enorm{u-u_h}{\Omega}+
  \norm{\phi-\phi_h}{\VV}
  \leq C_{\rm rel} \frac{1}{\min\left\{ \varepsilon, C_{\rm harm}\right\}} 
  \left( \sum_{T\in \TT}\eta_T^2 \right)^{1/2}
\end{align}
with 
 \begin{align*}  
 C_{\rm harm}=\frac{1}{2}  \left[  
          (1-\varepsilon)\alpha_{\min} +1  
          - \sqrt{ ((1-\varepsilon)\alpha_{\min}-1)^2+C_{\KK}}
        \right] 
\end{align*}
If we replace $\AA_V$ by $\AA^{up}_V$~\cref{eq:fvmbilinearup} in~\cref{eq:fvm}
we get the robust upper bound
\begin{align}
 \label{eq:robreliabilityup}
  \enorm{u-u_h}{\Omega}+
  \norm{\phi-\phi_h}{\VV}
  \leq C_{\rm rel} \frac{1}{\min\left\{ \varepsilon, C_{\rm harm}\right\}} 
  \left( \sum_{T\in \TT}\big(\eta_T^2 +\eta^2_{T,{up}}\big) \right)^{1/2}.
\end{align}
\end{theorem}

\begin{remark}\label{rem:robreliability}
 The constant $1/\min\left\{ \varepsilon, C_{\rm harm}\right\}$ needs some discussion.
 First we note that if $\alpha_{\min}> C_{\KK}/(4(1-\varepsilon))$ then $C_{\rm harm}> 0$
 and
  if $\alpha_{\min}\to C_{\KK}/(4(1-\varepsilon))$ then $C_{\rm harm}\to 0$ (monotone).
  Thus, if we want to guarantee $1/\min\left\{ \varepsilon, C_{\rm harm}\right\}=1/\varepsilon$,
  we have the constraint
 \begin{align*}
  \alpha_{\min}>\frac{4\varepsilon(1-\varepsilon)+C_{\KK}}{4(1-\varepsilon)^2}.  
 \end{align*}
 Note that the contraction constant $C_{\KK}\in[1/2,1)$ depends on the shape of $\Omega$.
 For example, if we set $\varepsilon=1/10$ and pick $C_{\KK}=1$ (worst case) \cref{th:reliabilityfvem}
 holds for $\alpha_{\min}>0.4198$. Thus the reliability constant is in fact $10 C_{\rm rel}$,
 which is robust with respect to the jumping diffusion $\alpha$, $\b$ and $\r$.
\end{remark}

\begin{proof}
 Let us write $e:=u-u_h\in H^1(\Omega)$, $\delta:=\phi-\phi_h\in H^{-1/2}(\Gamma)$ for the errors.
 Some standard transformations,~\cref{eq:FEMBEM} and integration by parts lead to 
 \begin{align*}
 	\BB((e,\delta);(e,\delta))&=  
  \spl{f}{e}{\Omega} +\spe{t_0}{e}{\Gamma}+\spe{\delta}{(1/2-\KK)u_0}{\Gamma}\\
  &\qquad-\big(
   \AA(u_h,e)-\spe{\phi_h}{e}{\Gamma}
       +\spe{\delta}{(1/2-\KK)u_h}{\Gamma}+\spe{\delta}{\VV\phi_h}{\Gamma}\big)\\
   &=\sum_{T\in\TT}\int_T Re\,dx+\sum_{E\in\EE}\int_E Je\,dx\\
   &\qquad+\spe{\delta}{(1/2-\KK)u_0}{\Gamma}-\spe{\delta}{(1/2-\KK)u_h}{\Gamma}-\spe{\delta}{\VV\phi_h}{\Gamma}
 \end{align*}
 For the sums with $R$ and $J$ we use as in~\cite{Erath:2013-1} the $L^2$ orthogonality~\cref{eq:proportho}
 with $e_h^*=\IIh^*e_h$ where $e_h=\IIh e$ and add $e_h-e_h$.
  Then, Cauchy-Schwarz inequality, 
 the use of the robust estimates~\cref{eq:resfemin,eq:edgeJfvmin}
 and the localization~\eqref{eq:p12estimate}, see also \cref{eq2:fvm}, lead to
 \begin{align*}
  \BB((e,\delta);(e,\delta))&
  \lesssim
      \left[\left(\sum_{T\in\TT}\mu_T^2
       \norm{R}{L^2(T)}^2\right)^{1/2}+
       \left(\sum_{E\in\EE}\alpha_E^{-1/2}\mu_E
       \norm{J}{L^2(E)}^2\right)^{1/2}\right]\enorm{e}{\Omega}\\
  &\quad+\left(\sum_{E\in\EEr} h_E\norm{\nabla_\Gamma\left((1/2-\KK)(u_0-u_h)
       -\VV \phi_h\right)}{L^2(E)}^2\right)^{1/2}
       \norm{\delta}{H^{-1/2}(\Gamma)}
 \end{align*}
 Applying the Cauchy-Schwarz inequality again and the robust estimate~\cref{eq:robustelipt}
 proves the first assertion~\cref{eq:robreliability}.
 To prove~\cref{eq:robreliabilityup} we can use~\cite[Lemma 5.1, Lemma 5.2]{Erath:2013-1}
\end{proof}

\subsubsection{Non-robust reliable error estimator}

We can also give a non-robust error estimator, which can easily be defined for a diffusion matrix
$\A$ and is less restricting than the robust estimator.
\begin{align}
 \label{eq:refindicator}
\begin{split}
  \eta_T^2&:=h_T^2\norm{R}{L^2(T)}^2+
  \frac{1}{2}\sum_{E\in\EEi\cap\EEt}h_E\norm{J}{L^2(E)}^2
 +\sum_{E\in\EEr\cap\EEt}h_E\norm{J}{L^2(E)}^2\\
  &\quad+\sum_{E\in\EEr\cap\EEt}h_E\norm{\nabla_\Gamma 
  \big((1/2-\KK)(u_0-u_h)-\VV\phi_h\big)}{L^2(E)}^2
\end{split}
\end{align}
for all $T\in\TT$.

\begin{theorem}[Reliability]\label{th:reliabilitynonrobust}
Let us assume $\lambda_{\min} - \frac{1}{4} C_{\KK} > 0$.
Then, there is a constant $C_{\rm rel}>0$ which depends on the model data and on the shape of the
elements in $\TT$ but not on the size or the number of elements such that
\begin{align*}
  \norm{u-u_h}{H^1(\Omega)}+
  \norm{\phi-\phi_h}{H^{-1/2}(\Gamma)}
  \leq C_{\rm rel} \left( \sum_{T\in \TT}\eta_T^2 \right)^{1/2}.
\end{align*}
\end{theorem}

\begin{proof}
We will only sketch the proof as it mostly follows the lines above. 
A similar proof in the case of FEM-BEM coupling
with $\b=(0,0)^T$ and $\r=0$ can be found in~\cite{Aurada:2013-1}. 
Let us write $e:=u-u_h\in H^1(\Omega)$, $\delta:=\phi-\phi_h\in H^{-1/2}(\Gamma)$ for the errors.
Instead of using the robust estimate~\cref{eq:robustelipt}
we use the ellipticity of the equivalent stabilized bilinear form of~\cite{Erath:2015-2}, i.e.,
there holds
\begin{align}
 \label{eq:ellipticity}
 \norm{e}{H^1(\Omega)}+
 \norm{\delta}{H^{-1/2}(\Gamma)}\lesssim \BB((e,\delta);(e,\delta))
 +\beta\big(\spe{1}{(1/2+\KK)e+\VV\delta}{\Gamma}\big)=\BB((e,\delta);(e,\delta)).
\end{align}
The last step follows directly from the second coupling equations~\cref{eq2:weakfem} and~\cref{eq2:fvm}.
Note that the stabilization term ($\beta=1$) is only needed if $(\div\b)/2+\r=0$ 
almost everywhere in $\Omega$ (otherwise $\beta=0$).
As in the proof of the robust error estimator we arrive at
\begin{align*}
 \norm{e}{H^1(\Omega)}+
 \norm{\delta}{H^{-1/2}(\Gamma)}\lesssim\BB((e,\delta);(e,\delta))&=  
  \sum_{T\in\TT}\int_T Re\,dx+\sum_{E\in\EE}\int_E Je\,dx+\spe{\delta}{(1/2-\KK)u_0}{\Gamma}\\
   &\qquad-\spe{\delta}{(1/2-\KK)u_h}{\Gamma}-\spe{\delta}{\VV\phi_h}{\Gamma}.
 \end{align*}
Again, we add $e_h-e_h$ with $e_h=\IIh e$,
 where $\IIh$ can be the standard Cl\'ement nodal interpolant~\cref{eq:intopaffine} 
  in the sense of~\cite{Clement:1975-1},
  and use the $L^2$ orthogonality~\cref{eq:proportho} with $e^*_h=\IIh^*e_h$.
The resulting terms can 
then be estimated by means of the Cauchy-Schwarz inequality, Cl\'ement type interpolation 
estimates~\cite{Clement:1975-1},
and estimates of the piecewise constant
nodal interpolation operator $\IIh^*$; see~\cite[Lemma 4.1]{Erath:2012-1}.
\end{proof}

\begin{remark}
 The ellipticity~\cref{eq:ellipticity} only holds for 
 $\alpha_{\min}> C_{\KK}/4$; see
 \cref{rem:robreliability} for the robust estimate. The ellipticity
 constant depends on the model data and the boundary $\Gamma$ and can become very small.
 However, 
 this seems to be a theoretical bound; see~\cite{Erath:2015-2}.
\end{remark}

\subsection{Efficiency}

Following~\cite{Verfurth:book-1996}, the analysis to prove efficiency for 
our residual based a~posteriori error estimator needs some more regularity
on the solution and the data. Thus we assume
$u|_\Gamma\in H^1(\Gamma)$, $\phi\in L^2(\Gamma)$, $u_0\in H^1(\Gamma)$
and $t_0\in L^2(\Gamma)$.
The idea is to use so called bubble functions and an edge lifting operator, which imply 
some inverse estimates for polynomial functions.
To get a lower bound in the energy (semi)norm~\cref{eq:energynorm} 
the inverse estimates are based on bubble functions on a squeezed element; see, 
e.g.,~\cite{Verfurth:1998-1,Erath:2010-phd}.

To get a lower bound for the terms with the boundary integral operators $\VV$
and $\KK$ of~\cref{eq:refindicatorrobust} we require that $\TT$ is a quasi-uniform 
mesh on the boundary
$\Gamma$. That means, the ratio of the longest edge in $\EEr$ to the shortest edge in $\EEr$
for a sequence of meshes is bounded by a constant, which does not depend on the 
size of the elements. Furthermore, there is only a global upper bound available. We stress 
that even for FEM-BEM residual estimators there is no better efficiency result available in the literature.
For more details we refer to~\cite[Section 6]{Erath:2010-phd} and~\cite{Erath:2013-1}.

\begin{lemma}
\label{lem:bem}
	  Let $\widetilde{\xi}_h\in\SS^1(\EEr)$ 
	  be the nodal interpolant of $u|_\Gamma$ and $\overline{\phi}\in\PP^0(\EEr)$ 
	  the $\EEr$-piecewise integral mean
	  of $\phi$, $h_{\Gamma,max}:=\max_{E\in\EEr}h_E$ and $h_{\Gamma,min}:=\min_{E\in\EEr}h_E$. 
	  Then, there holds the global estimate
	\begin{align*}
	&\sum_{E\in\EEr}h_E^{1/2}\norm{\nabla_\Gamma 
		  \big((1/2-\KK)(u_0-u_h)-\VV\phi_h\big)}{L^2(E)}\\
	  &\qquad\leq\sum_{E\in\EEr}h_E^{1/2}\Big(\norm{\nabla_\Gamma
	  \big((1/2+\KK)(u|_\Gamma-u_h)\big)}{L^2(\Gamma)} 
	    +\norm{\nabla_\Gamma\VV(\phi-\phi_h)}{L^2(\Gamma)}\Big) \\
	  &\qquad\lesssim
	  \frac{h_{\Gamma,max}^{1/2}}{h_{\Gamma,min}^{1/2}}\Big(
	  \norm{u-u_h}{H^{1}(\Omega)}+\norm{\phi-\phi_h}{H^{-1/2}(\Gamma)}\\
	  &\qquad\qquad\qquad\quad
	  +h_{\Gamma,max}^{1/2}\big(\norm{u|_\Gamma-\widetilde{\xi}_h}{H^{1}(\Gamma)}
	  +\norm{\phi-\overline{\phi}}{L^{2}(\Gamma)}\big)\Big).
	\end{align*}
\end{lemma}
\begin{proof}
	The first inequality uses the relation $(1/2-\KK)u_0=(1/2-\KK)u|_\Gamma+\VV\phi\in H^1(\Gamma)$
	and the triangle inequality. 
	With $\norm{u|_\Gamma-u_h}{H^{1/2}(\Gamma)}\lesssim\norm{u-u_h}{H^{1}(\Omega)}$
	the second estimate follows directly from~\cite{Erath:2010-phd}.
\end{proof}

Finally we are able to formulate an efficiency statement
for our a~posteriori error estimator, i.e., $\eta$ is a lower bound of the error.
\begin{theorem}[Efficiency]\label{th:efficiency}
If $\TT$ is a quasi-uniform mesh on the boundary $\Gamma$, we get an inverse inequality
to the reliability \cref{th:reliabilityfvem}; i.e., the a~posteriori
estimate is sharp up to higher order terms. The quantities generated 
through the interior problem approximation are even
locally efficient without any restriction on the boundary mesh. 
\end{theorem}

\begin{proof}
 Lower estimates for the contributions with the residual $R$ and the jump terms
 $J$ of the refinement indicator $\eta_T$ in~\cref{eq:refindicatorrobust}
 can be found in~\cite[Claim~1 -- Claim~4]{Erath:2013-1}. 
 For terms with the boundary integral operators $\VV$
 and $\KK$ of~\cref{eq:refindicatorrobust} we apply \cref{lem:bem}.
 Then, the efficiency 
 \cref{th:efficiency} can be shown up to higher order terms.
 See also \cref{rem:robust}.
\end{proof}

\begin{remark}
\label{rem:robust}
For a detailed discussion of the higher order terms on the right-hand side of the lower bound
we refer to~\cite{Erath:2013-1,Erath:2010-phd}.
We note that the estimate depends on the local P\'eclet number,
which is a typical behaviour of such problems in the energy norm; 
see also the discussion about robustness in~\cite[Remark 6.1]{Erath:2013-1}
for the three field FVM-BEM coupling.
Obviously, \cref{th:efficiency} holds also for the inverse inequality of
the reliability
\cref{th:reliabilitynonrobust} with the non-robust refinement indicator~\cref{eq:refindicator}.
\end{remark}

 \section{Numerical results}
 \label{sec:numerics}

To verify the analytical findings and to show the strength of an adaptive refinement strategy,
we present three examples in two dimensions. 
The calculations were done in \textsc{Matlab} using some functions 
from the \textsc{Hilbert}-package~\cite{HILBERT:2013-1} for the matrices resulting 
from the integral operators $\VV$ and $\KK$. 
The arc-length derivative in the error estimator is estimated by a central difference quotient, thus by
\begin{align*}
\nabla_\Gamma v(x) \approx \frac{v(x_2)-v(x_1)}{|x_2-x_1|} 
\end{align*}
with $|x_2-x_1|=h_E/20$ and $x=(x_2+x_1)/2$.
If there is any convection involved we will use the full upwind scheme \cref{eq:upwind}
and replace the bilinear form $\AA_V$ in \cref{eq1:fvm} by $\AA_V^{up}$ defined in \cref{eq:fvmbilinearup}.
The error will be denoted by $E_h$, defined as
\begin{align*}
 E_h:=\enorm{u-u_h}{\Omega}+\norm{\phi-\phi_h}{\VV}
 \end{align*}
(recall that $\norm{\cdot}{H^{-1/2}}\sim\norm{\cdot}{\VV}$) and the error estimator $\eta$ 
is given by the sum of the indicators
from \cref{eq:refindicatorrobust,eq:refindicatorup} or \cref{eq:refindicator}
\begin{align*}
 \eta:=\sum_{T\in\TT} \left(\eta_T^2 (+\eta_{T,{up}}^2)\right)^{1/2},
\end{align*}
where the $\eta_{T,{up}}^2$ part is added when an upwind stabilization is used.

We will apply the refinement algorithm introduced in~\cite{Doerfler:1996-1}
with the following marking criterion:
let $\theta \in (0,1)$, then at the refinement step $k$ 
choose $\mathcal{M}^{(k)}\subset\TT^{(k)}$ with minimal cardinality such that
\begin{align*}
\sum_{T\in\mathcal{M}^{(k)}} \left(\eta_T^2 (+\eta_{T,{up}}^2)\right) \geq \theta \sum_{T\in\TT^{(k)}} \left(\eta_T^2 (+\eta_{T,{up}}^2)\right).
\end{align*} 
The elements in the subset $\mathcal{M}^{(k)}$ will then be refined by a red-green-blue refinement
strategy,
which leads to refined mesh $\TT^{(k+1)}$; 
see also~\cite{Verfurth:book-1996}). Therefore, the shape regularity constant
is bounded in all of our examples.
We choose $\theta=1/2$ for adaptive mesh refinement, $\theta=1$ means uniform refinement.
The regular initial triangulation $\TT^{(0)}$ will always have triangles with approximately the same size.

%

\begin{figure}
 \begin{center}
\begin{psfrags}%
\psfragscanon%
%
\psfrag{s01}[t][t]{\color[rgb]{0.15,0.15,0.15}\setlength{\tabcolsep}{0pt}\begin{tabular}{c}number of elements\end{tabular}}%
\psfrag{s02}[b][b]{\color[rgb]{0.15,0.15,0.15}\setlength{\tabcolsep}{0pt}\begin{tabular}{c}error and estimator\end{tabular}}%
\psfrag{s03}[l][l]{\color[rgb]{0,0,0}\setlength{\tabcolsep}{0pt}\begin{tabular}{l}\small $1/1$\end{tabular}}%
\psfrag{s04}[l][l]{\color[rgb]{0,0,0}\setlength{\tabcolsep}{0pt}\begin{tabular}{l}\small      $1/2$     \end{tabular}}%
\psfrag{s05}[l][l]{\color[rgb]{0,0,0}\setlength{\tabcolsep}{0pt}\begin{tabular}{l}\small$1/1$\end{tabular}}%
\psfrag{s06}[l][l]{\color[rgb]{0,0,0}\setlength{\tabcolsep}{0pt}\begin{tabular}{l}\small      $1/2$    \end{tabular}}%
\psfrag{s07}[l][l]{\color[rgb]{0,0,0}\setlength{\tabcolsep}{0pt}\begin{tabular}{l} \small     $1/3$     \end{tabular}}%
\psfrag{s08}[l][l]{\color[rgb]{0,0,0}\setlength{\tabcolsep}{0pt}\begin{tabular}{l}1/1\end{tabular}}%
%
\psfrag{l01}[l][l]{\color[rgb]{0,0,0}\setlength{\tabcolsep}{0pt}\begin{tabular}{c}$\eta$ (ada)\end{tabular}}%
\psfrag{l02}[l][l]{\color[rgb]{0,0,0}\setlength{\tabcolsep}{0pt}\begin{tabular}{c}$\eta$ (uni)\end{tabular}}%
\psfrag{l03}[l][l]{\color[rgb]{0,0,0}\setlength{\tabcolsep}{0pt}\begin{tabular}{l}$E_h$ (ada)\end{tabular}}%
\psfrag{l04}[l][l]{\color[rgb]{0,0,0}\setlength{\tabcolsep}{0pt}\begin{tabular}{l}$E_h$ (uni)\end{tabular}}%

\color[rgb]{0.15,0.15,0.15}%
%
\psfrag{x01}[t][t]{\small $10^{0}$}%
\psfrag{x02}[t][t]{\small $10^{2}$}%
\psfrag{x03}[t][t]{\small $10^{4}$}%
\psfrag{x04}[t][t]{\small $10^{6}$}%
\psfrag{x05}[t][t]{\small $10^{8}$}%
%
\psfrag{v01}[r][r]{\small $10^{-4}$}%
\psfrag{v02}[r][r]{\small $10^{-3}$}%
\psfrag{v03}[r][r]{\small $10^{-2}$}%
\psfrag{v04}[r][r]{\small $10^{-1}$}%
\psfrag{v05}[r][r]{\small$10^{0}$}%
\psfrag{v06}[r][r]{\small$10^{1}$}%
%
\resizebox{0.7\textwidth}{!}{\includegraphics{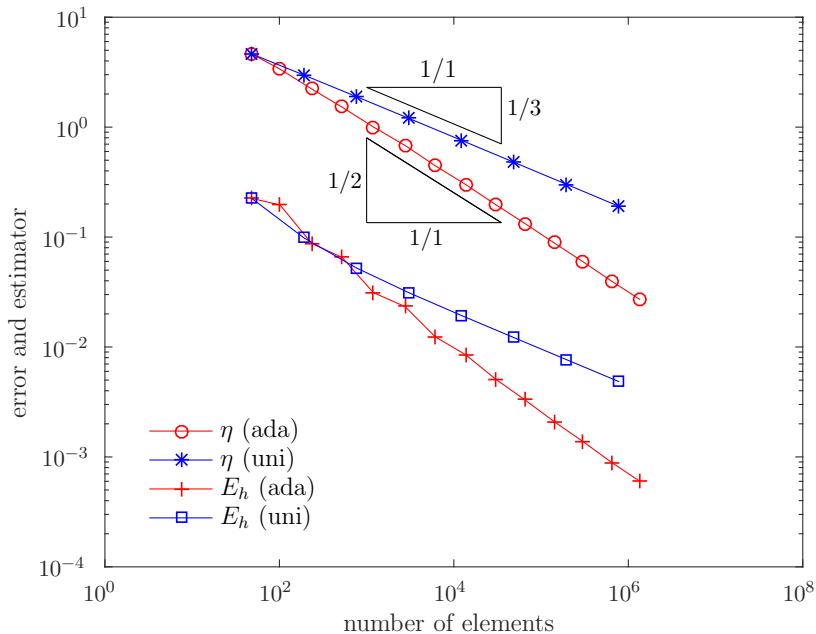}}%
\end{psfrags}%
\end{center}
\caption{\label{fig:sinError}
 The error estimator $\eta$ and the error $E_h$  
in the case of uniform (uni) and adaptive (ada) mesh refinement
for the example in~\cref{subsec:bsp1}. The recovery of the optimal convergence rate
in the adaptive case can be seen.}
\end{figure}

 \begin{figure}
 \centering
	\subfigure[$\#\TT^{(4)}=1248$.]{
	\includegraphics[width=0.35\textwidth]{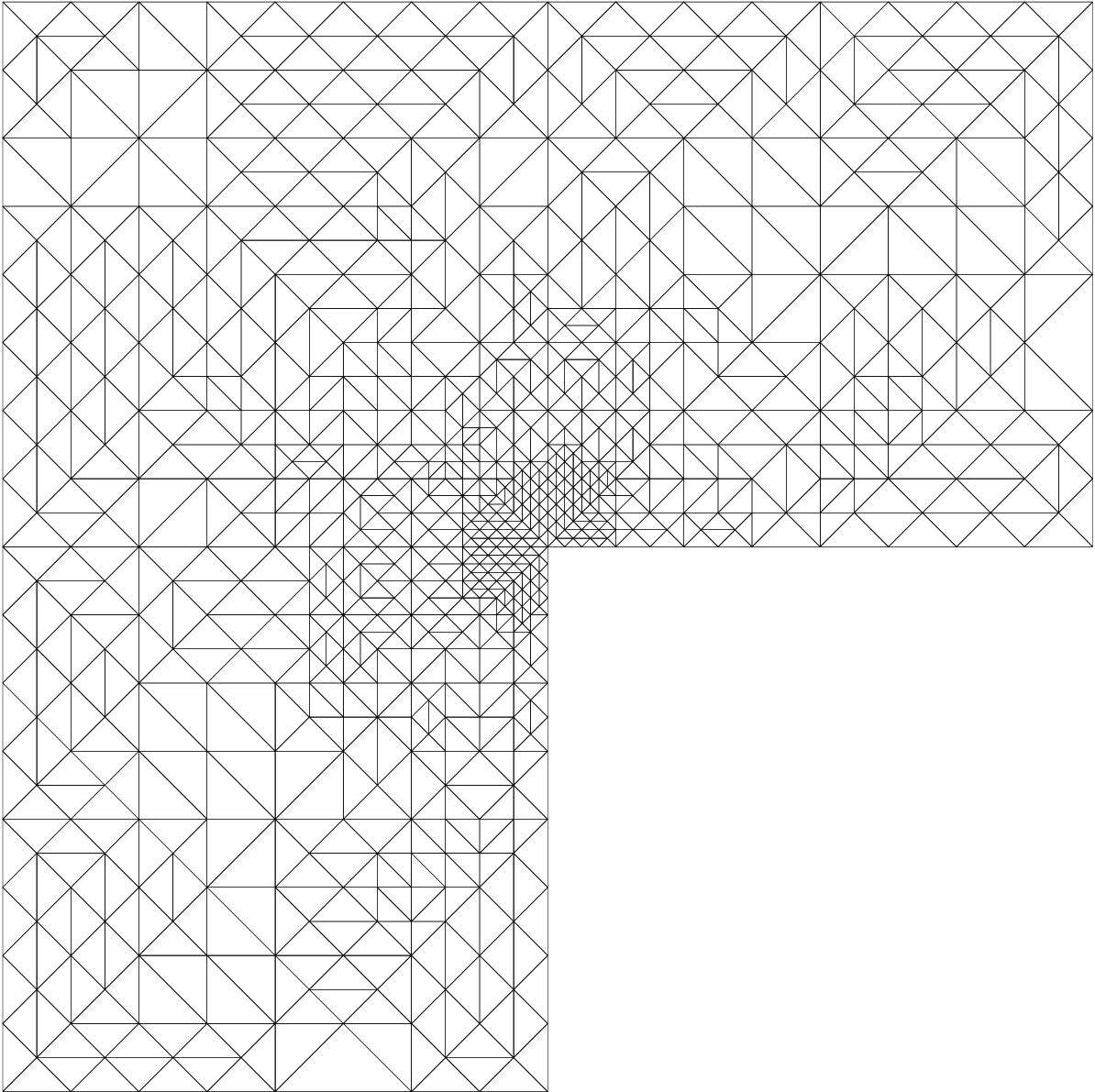}}
	\hspace{0.1\textwidth}
	\subfigure[$\#\TT^{(6)}=6314$.]{
	\includegraphics[width=0.35\textwidth]{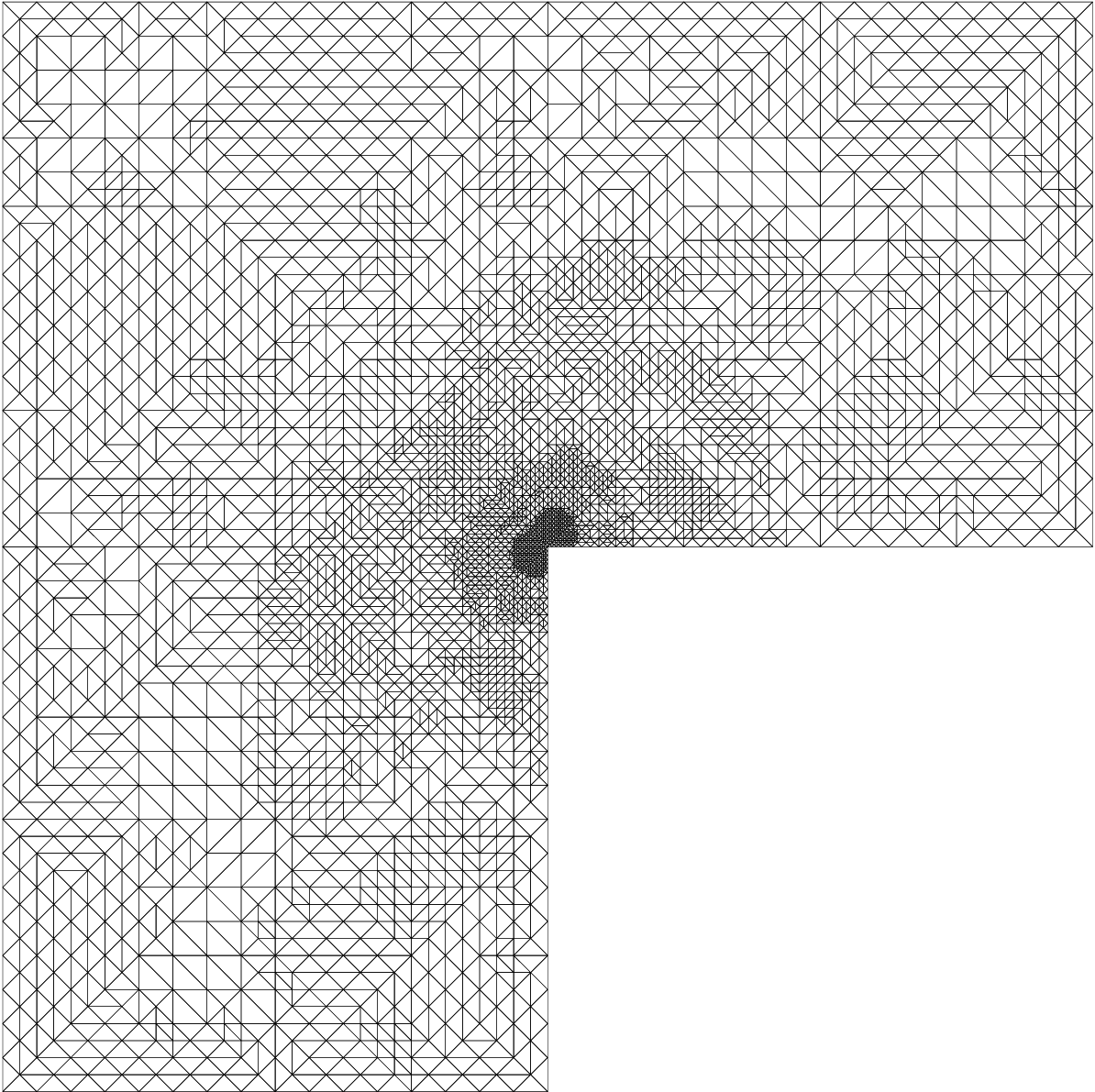}}
 \caption{\label{fig:sinMesh}
 Two adaptively generated meshes $\TT^{(4)}$ and $\TT^{(6)}$ of the fourth and sixth refinement step 
 of the example in~\cref{subsec:bsp1}. The refinement mainly takes place around the singularity.}
\end{figure}

\subsection{Diffusion problem on an L-shaped domain}
 \label{subsec:bsp1}

As a first test we consider a purely diffusive problem of model problem \cref{eq:model}, 
i.e., without any convection or reaction, $\b=(0,0)^T$ and $c=0$,
but a diffusion matrix. 
We want to have a specific solution to this problem. 
So we prescribe the coefficients and right-hand side appropriately.
The domain will be L-shaped, i.e., $\Omega= (-1/4, 1/4)^2 \setminus [0, 1/4]\times [-1/4,0]$. 
We prescribe a function that has a singularity in the corner $(0,0)$ 
of our domain (the gradient tends to infinity at this point).
The analytical solution in the interior domain $\Omega$ will then read (in polar 
coordinates $(x_1,x_2) = r(\cos\varphi, \sin\varphi)$ with $r\in\R_+$ and $\varphi\in [0,2\pi)$)
\begin{align*}
u(x_1,x_2) = r^{2/3} \sin(2\varphi/3),
\end{align*}
and in the exterior domain $\Omega_e$
\begin{align*}
u_e(x_1,x_2)=\log \sqrt{(x_1+0.125)^2+(x_2-0.125)^2}.
\end{align*}
Furthermore, we will fix the diffusion matrix to  
\begin{align*}
	\A=  \left ( \begin{array}{rr}
  10+\cos x_1 & 160\,x_1 x_2 \\
  160\,x_1 x_2 & 10+\sin x_2
  \end{array}\right).
\end{align*}
We compute $f$ and the jumps $t_0$ and $u_0$ according to the formulas.

Note that 
the function $u$ is not in $H^2(\Omega)$ and thus the optimal convergence 
rate of $\OO(h)$ for uniform mesh refinement~\cite{Erath:2015-2}) cannot be obtained.
The notation $\OO(h)$, where $h$ is the minimal
diameter of an element of the mesh, is a bit misleading in the adaptive case.
Therefore, we consider $\OO(N^{-p})$, where $N$ is the number of elements and $p\in\R^+$, 
which is equivalent to $\OO(h^{p})$ for uniform mesh refinement.
\Cref{fig:sinError} shows the error and error estimator for uniform and adaptive mesh
refinement. For uniform refinement we observe the reduced convergence order $\OO(N^{-1/3})$, whereas
with our adaptive strategy we can recover the optimal rate $\OO(N^{-1/2})$. This classical
benchmark result matches observations from the literature.
Note that in both cases the estimator is reliable and efficient.
\Cref{fig:sinMesh} shows two adaptively generated meshes, $\TT^{(4)}$ and $\TT^{(6)}$, 
generated from a start mesh $\TT^{(0)}$ with $48$ elements.
The refinement mainly takes place around the singularity.


 \begin{figure}
 \begin{center}
\subfigure[\label{subfig:tanherror}Error estimator $\eta$ and energy error $E_h$.]{
 \begin{psfrags}%
\psfragscanon%
%
\psfrag{s02}[l][l]{\color[rgb]{0,0,0}\setlength{\tabcolsep}{0pt}\begin{tabular}{l}      \small $1/2$     \end{tabular}}%
\psfrag{s03}[t][t]{\color[rgb]{0.15,0.15,0.15}\setlength{\tabcolsep}{0pt}\begin{tabular}{c}number of elements\end{tabular}}%
\psfrag{s04}[l][l]{\color[rgb]{0,0,0}\setlength{\tabcolsep}{0pt}\begin{tabular}{l}\small $1/1$\end{tabular}}%
\psfrag{s05}[b][b]{\color[rgb]{0.15,0.15,0.15}\setlength{\tabcolsep}{0pt}\begin{tabular}{c}error\end{tabular}}%
%
\psfrag{l01}[l][l]{\color[rgb]{0,0,0}\setlength{\tabcolsep}{0pt}\begin{tabular}{c}$\eta$ (ada)\end{tabular}}%
\psfrag{l02}[l][l]{\color[rgb]{0,0,0}\setlength{\tabcolsep}{0pt}\begin{tabular}{c}$\eta$ (uni)\end{tabular}}%
\psfrag{l03}[l][l]{\color[rgb]{0,0,0}\setlength{\tabcolsep}{0pt}\begin{tabular}{c}$E_h$ (ada)\end{tabular}}%
\psfrag{l04}[l][l]{\color[rgb]{0,0,0}\setlength{\tabcolsep}{0pt}\begin{tabular}{c}$E_h$ (uni)\end{tabular}}%

\color[rgb]{0.15,0.15,0.15}%
%
\psfrag{x01}[t][t]{\small $10^{0}$}%
\psfrag{x02}[t][t]{\small $10^{2}$}%
\psfrag{x03}[t][t]{\small $10^{4}$}%
\psfrag{x04}[t][t]{\small $10^{6}$}%
\psfrag{x05}[t][t]{\small $10^{8}$}%
%
\psfrag{v01}[r][r]{\small $10^{-2}$}%
\psfrag{v02}[r][r]{\small $10^{-1}$}%
\psfrag{v03}[r][r]{\small $10^{0}$}%
\psfrag{v04}[r][r]{\small $10^{1}$}%
\psfrag{v05}[r][r]{\small $10^{2}$}%
%
\resizebox{0.5\textwidth}{!}{\includegraphics{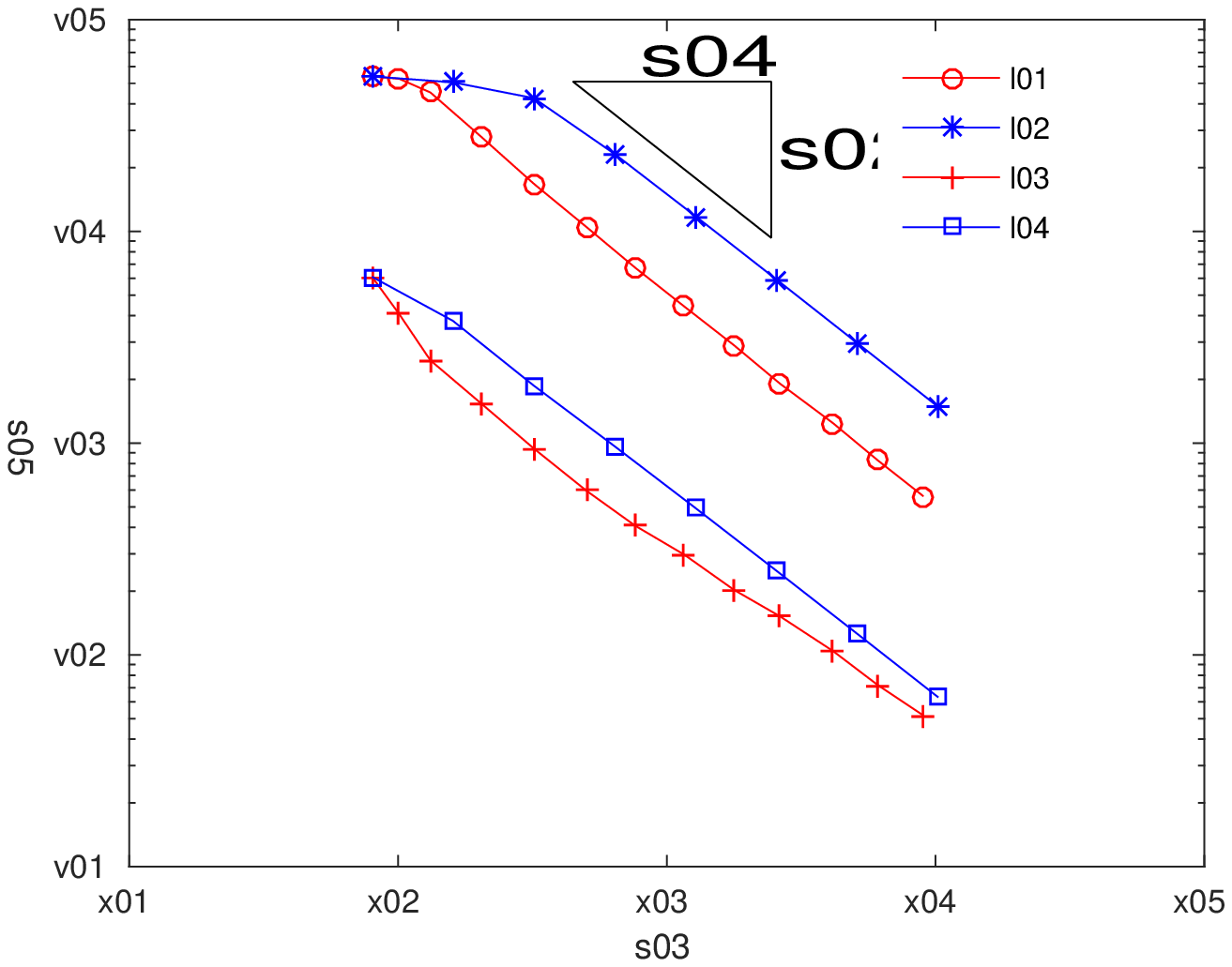}}%
\end{psfrags}}%
\subfigure[\label{subfig:efficiencyIndex}Efficiency index $\eta/E_h$.]{
\begin{psfrags}%
\psfragscanon%
%
\psfrag{s03}[t][t]{\color[rgb]{0.15,0.15,0.15}\setlength{\tabcolsep}{0pt}\begin{tabular}{c}number of elements\end{tabular}}%
\psfrag{s04}[b][b]{\color[rgb]{0.15,0.15,0.15}\setlength{\tabcolsep}{0pt}\begin{tabular}{c}efficiency index\end{tabular}}%
%
\psfrag{l01}[l][l]{\color[rgb]{0,0,0}\setlength{\tabcolsep}{0pt}\begin{tabular}{c}$\b=(10x_1,0)^T$\end{tabular}}%
\psfrag{l02}[l][l]{\color[rgb]{0,0,0}\setlength{\tabcolsep}{0pt}\begin{tabular}{c}$\b=(100x_1,0)^T$\end{tabular}}%
\psfrag{l03}[l][l]{\color[rgb]{0,0,0}\setlength{\tabcolsep}{0pt}\begin{tabular}{c}$\b=(1000x_1,0)^T$\end{tabular}}%
\psfrag{l04}[l][l]{\color[rgb]{0,0,0}\setlength{\tabcolsep}{0pt}\begin{tabular}{c}$\b=(10000x_1,0)^T$\end{tabular}}%
\color[rgb]{0.15,0.15,0.15}%
%
\psfrag{x01}[t][t]{\small $10^{1}$}%
\psfrag{x02}[t][t]{\small $10^{2}$}%
\psfrag{x03}[t][t]{\small $10^{3}$}%
\psfrag{x04}[t][t]{\small $10^{4}$}%
\psfrag{x05}[t][t]{\small $10^{5}$}%
\psfrag{x06}[t][t]{\small $10^{6}$}%
%
\psfrag{v01}[r][r]{\small $10^{0}$}%
\psfrag{v02}[r][r]{\small $10^{1}$}%
\psfrag{v03}[r][r]{\small $10^{2}$}%
%
\resizebox{0.5\textwidth}{!}{\includegraphics{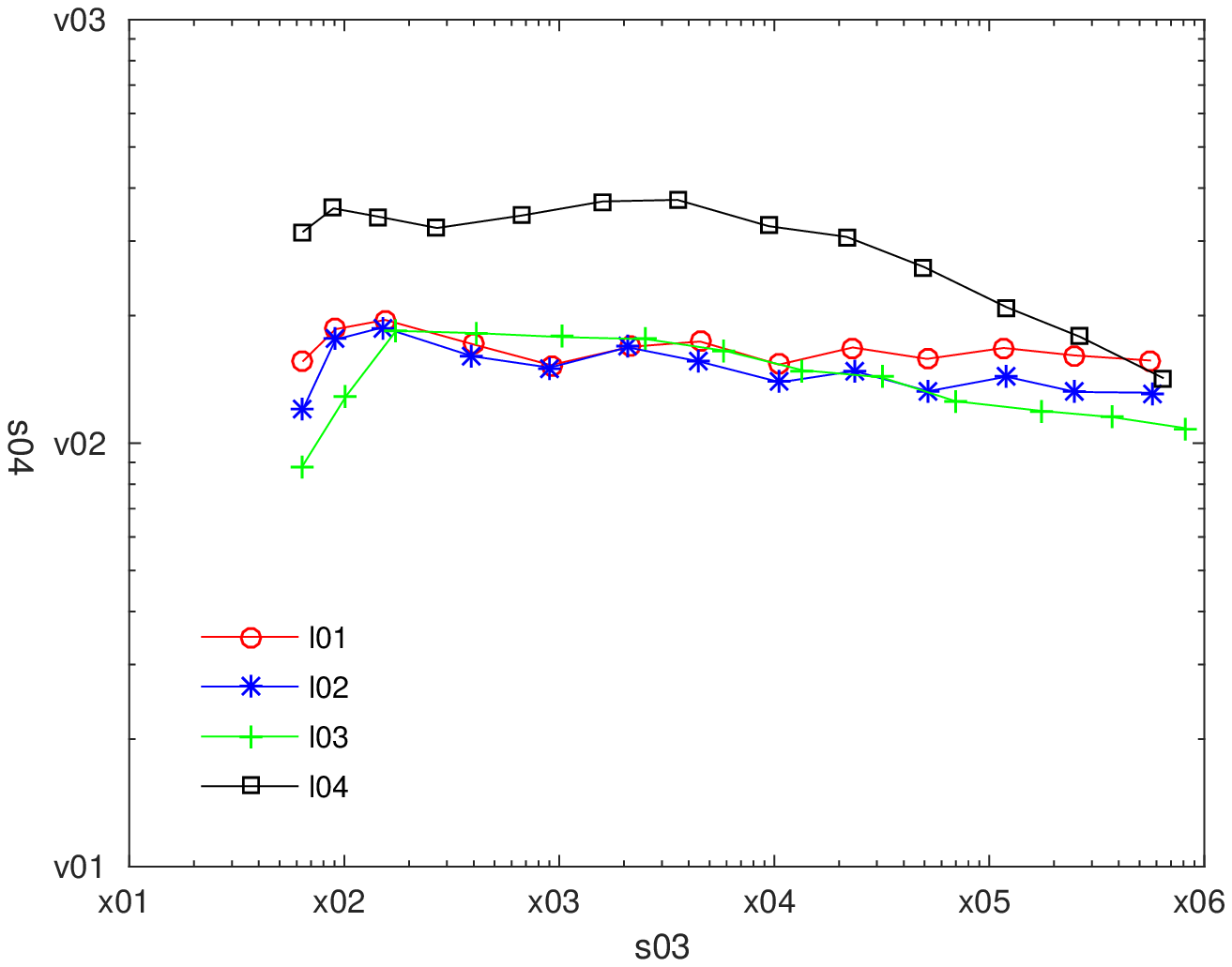}}%
\end{psfrags}}%
\end{center}
\caption{
In the left figure the error estimator $\eta$ and the error for the example in \cref{subsec:bsp2} are shown
in the case of uniform (uni) and adaptive (ada) mesh refinement. Because of the smoothness of the solution,
the convergence rate is optimal also with uniform refinement. In the right figure, 
the efficiency index for different values for $\b$ is shown.}
\end{figure}

 \begin{figure}
 \centering
	\subfigure[$\#\TT^{(4)}=1034$.]{
	\includegraphics[width=0.35\textwidth]{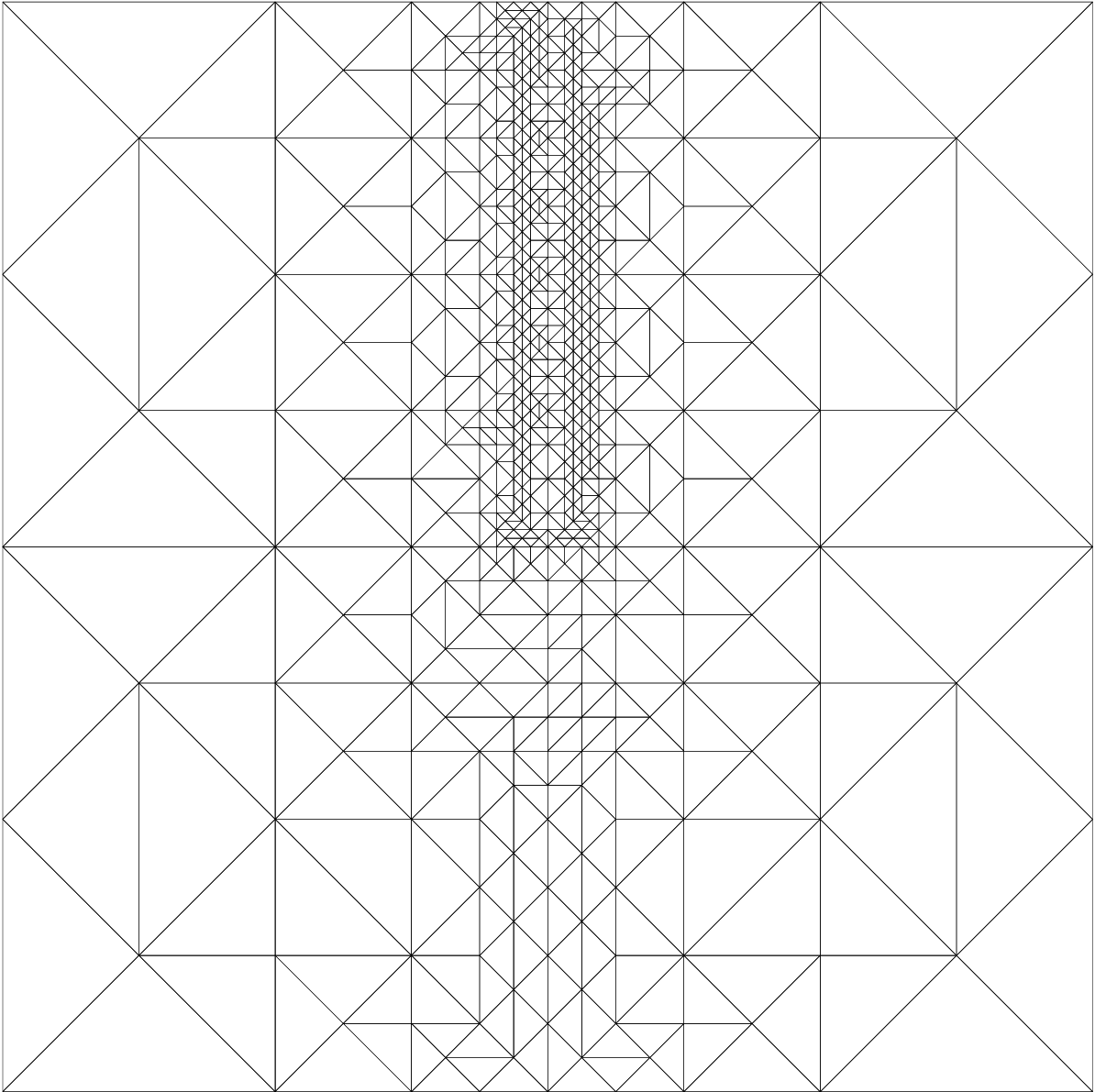}}
	\hspace{0.1\textwidth}
	\subfigure[$\#\TT^{(6)}=5846$.]{
	\includegraphics[width=0.35\textwidth]{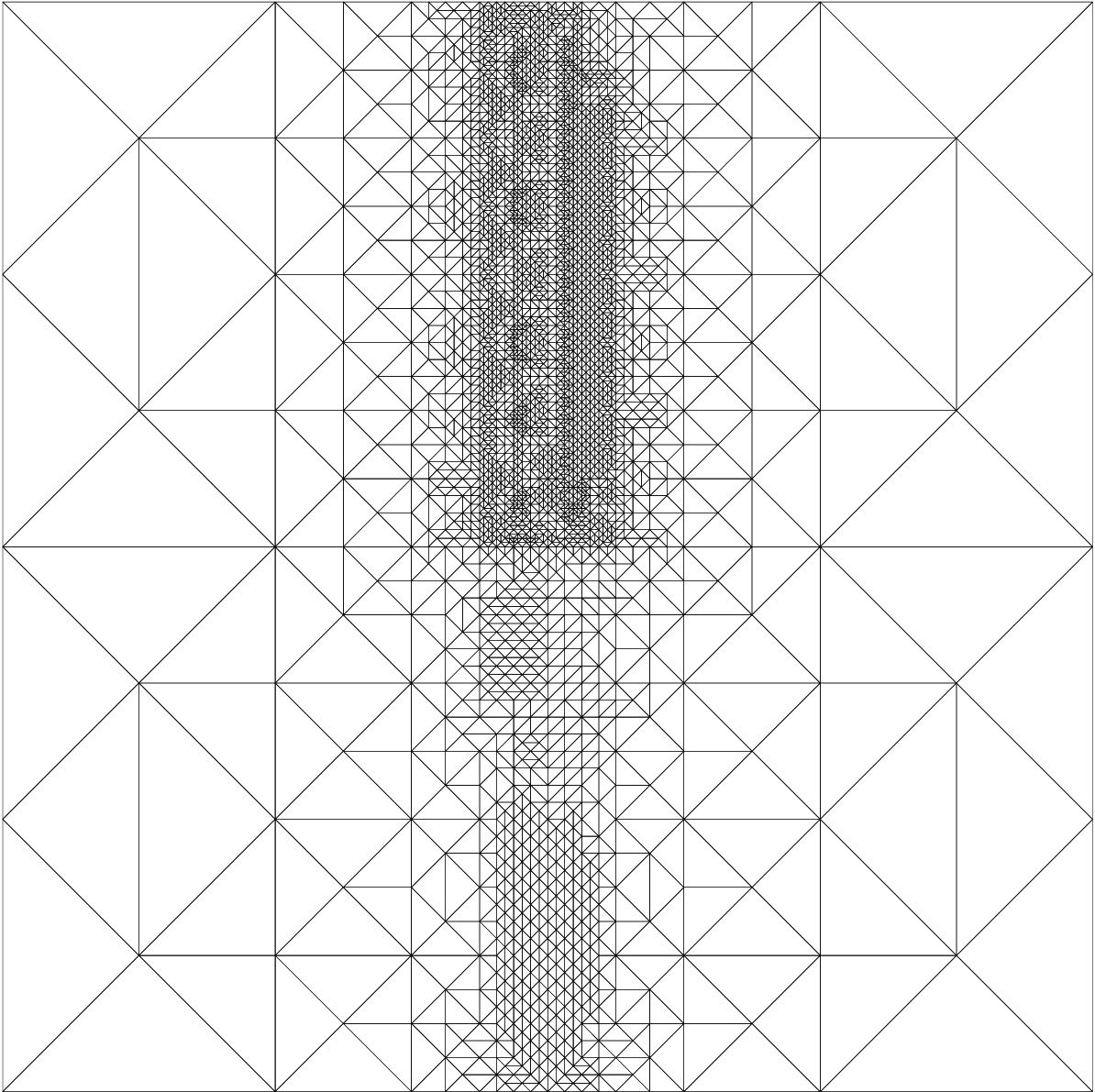}}
 \caption{\label{fig:tanhMesh}
 Two adaptively generated meshes $\TT^{(4)}$ and $\TT^{(6)}$ of the fourth and sixth refinement step
 of the example in~\cref{subsec:bsp2}. The refinement is mainly at the artificial shock of the 
 analytical solution.}
\end{figure}

 \subsection{Diffusion-convection problem}
 \label{subsec:bsp2}
 
For the next example we prescribe again a solution with known analytical 
properties for the model problem \cref{eq:model}. Now we choose $\Omega=(0,1/2)\times (0,1/2)$.
The solution in the interior domain $\Omega$ will be chosen to be
\begin{align*}
u(x_1,x_2)=0.5\left( 1-\tanh\left(\frac{0.25-x_1}{0.02}\right)\right),
\end{align*}
and the solution in the exterior domain $\Omega_e$ is similar as before, i.e., 
\begin{align*}
u_e(x_1,x_2)=\log \sqrt{(x_1-0.25)^2+(x_2-0.25)^2}.
\end{align*}
Thus, the interior solution has a simulated shock in the middle of the domain.
We choose the jumping diffusion coefficient as 
\begin{align*}
 \alpha=\begin{cases} 0.42\quad &\text{for}~x_2 < 0.25, \\
10 \quad &\text{for}~x_2 \geq 0.25, \end{cases}
\end{align*} 
the convection field $\b=(1000x_1,0)^T$ and the reaction coefficient $c=0$. 
So we have a convection dominated problem 
which will not yield a stable solution if we are not 
using any upwind stabilization \cref{eq:upwind}.
Because of that we will always use the full upwind scheme 
for this problem. The right-hand side $f$ and the jumps are calculated by means of 
the analytical solution.
Because of the smoothness of the interior and exterior solution 
we would expect an (optimal) convergence rate of $\OO(N^{-1/2})$ also for uniform mesh refinement. 
This can be seen in \cref{subfig:tanherror}.
For adaptive refinement we get $\OO(N^{-1/2})$ as well but the absolute value of the error
is actual smaller. Note that in both cases, uniform and adaptive mesh refinement, the estimator is
reliable and efficient.
The refinement (mainly) occurs where the function has its steepest gradient and is different for the
two values of the diffusion coefficient, see \cref{fig:tanhMesh} for 
the two meshes $\TT^{(4)}$ and $\TT^{(6)}$ generated from a start mesh with $64$ elements.
In \cref{subfig:efficiencyIndex} the efficiency index $\eta/E_h$, which measures
how many times we have overestimated the actual error,
is plotted for adaptive mesh refinement. We see indeed the robustness 
of our error estimator for $\b=\{(10x_1,0)^T;(100x_1,0)^T;(1000x_1,0)^T\}$. 
For $\b=(10000x_1,0)^T$, which is a very high convection dominated problem,
we observe the dependency of the local P\'eclet number, i.e., once we have resolved the shock region,
the efficiency constant convergences as well.

 \begin{figure}
 \begin{center}
 \begin{psfrags}%
\psfragscanon%
%
\psfrag{s02}[l][l]{\color[rgb]{0,0,0}\setlength{\tabcolsep}{0pt}\begin{tabular}{l}\small $2/5$     \end{tabular}}%
\psfrag{s03}[l][l]{\color[rgb]{0,0,0}\setlength{\tabcolsep}{0pt}\begin{tabular}{l}\small $1/1$\end{tabular}}%
\psfrag{s04}[l][l]{\color[rgb]{0,0,0}\setlength{\tabcolsep}{0pt}\begin{tabular}{l}\small $1/1$\end{tabular}}%
\psfrag{s06}[t][t]{\color[rgb]{0.15,0.15,0.15}\setlength{\tabcolsep}{0pt}\begin{tabular}{c}number of elements\end{tabular}}%
\psfrag{s07}[l][l]{\color[rgb]{0,0,0}\setlength{\tabcolsep}{0pt}\begin{tabular}{l}\small $1/2$     \end{tabular}}%
\psfrag{s08}[b][b]{estimator}%
%
\psfrag{l01}[l][l]{\color[rgb]{0,0,0}\setlength{\tabcolsep}{0pt}\begin{tabular}{c}$\eta$ (ada)\end{tabular}}%
\psfrag{l02}[l][l]{\color[rgb]{0,0,0}\setlength{\tabcolsep}{0pt}\begin{tabular}{c}$\eta$ (uni)\end{tabular}}%
\color[rgb]{0.15,0.15,0.15}%
%
\psfrag{x01}[t][t]{\small $10^{1}$}%
\psfrag{x02}[t][t]{\small $10^{2}$}%
\psfrag{x03}[t][t]{\small $10^{3}$}%
\psfrag{x04}[t][t]{\small $10^{4}$}%
\psfrag{x05}[t][t]{\small $10^{5}$}%
\psfrag{x06}[t][t]{\small $10^{6}$}%
%
\psfrag{v01}[r][r]{\small $10^{-2}$}%
\psfrag{v02}[r][r]{\small $10^{{-1}}$}%
\psfrag{v03}[r][r]{\small $10^{0}$}%
%
\resizebox{0.7\textwidth}{!}{\includegraphics{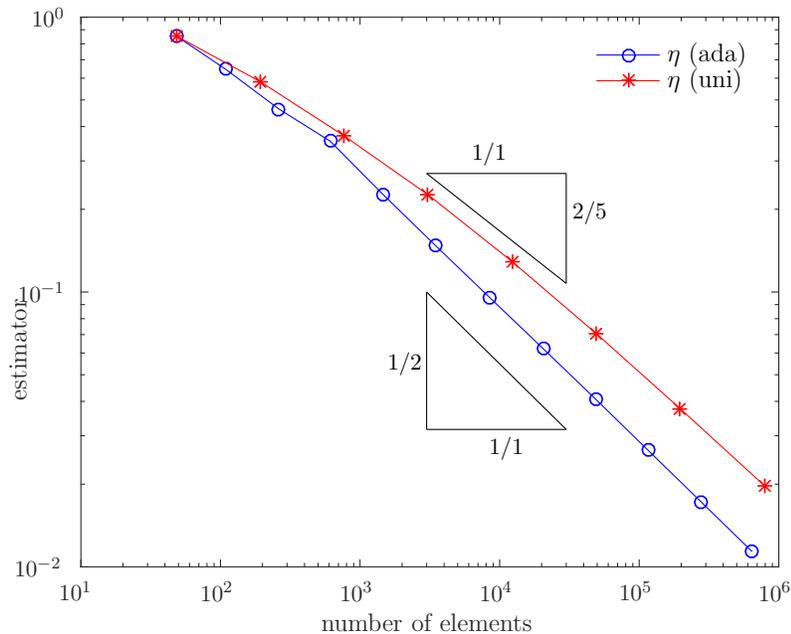}}%
\end{psfrags}%
\end{center}
\caption{\label{fig:practicalerror}
 The error estimators $\eta$ for the example in~\cref{subsec:bsp3} 
in the case of uniform (uni) and adaptive (ada) mesh refinement. Again we do not have the optimal convergence rate
in the uniform case.}
\end{figure}

 \begin{figure}
 \centering
	\subfigure[$\#\TT^{(4)}=1447$.]{
	\includegraphics[width=0.35\textwidth]{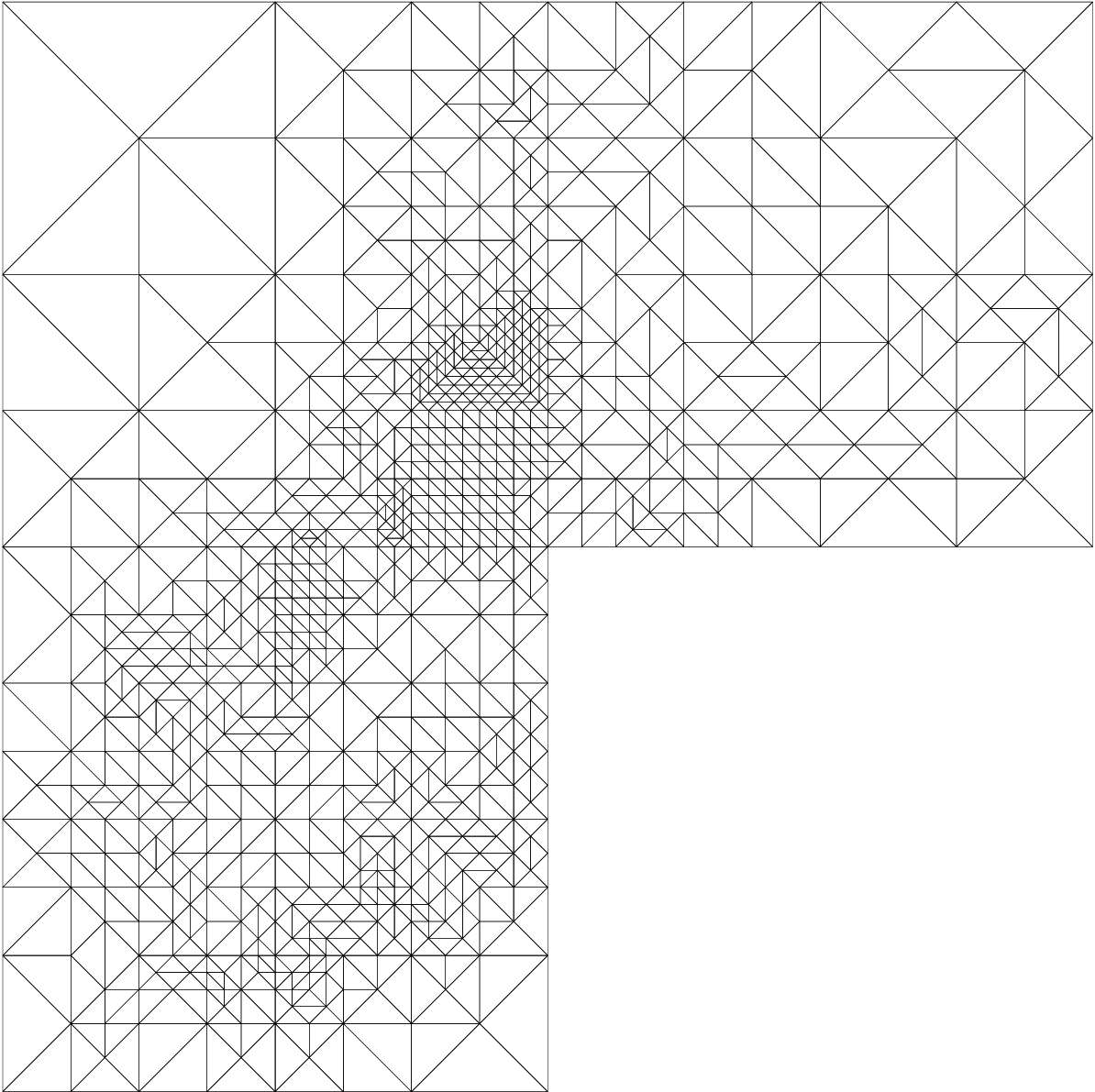}}
	\hspace{0.1\textwidth}
	\subfigure[$\#\TT^{(6)}=8451$.]{
	\includegraphics[width=0.35\textwidth]{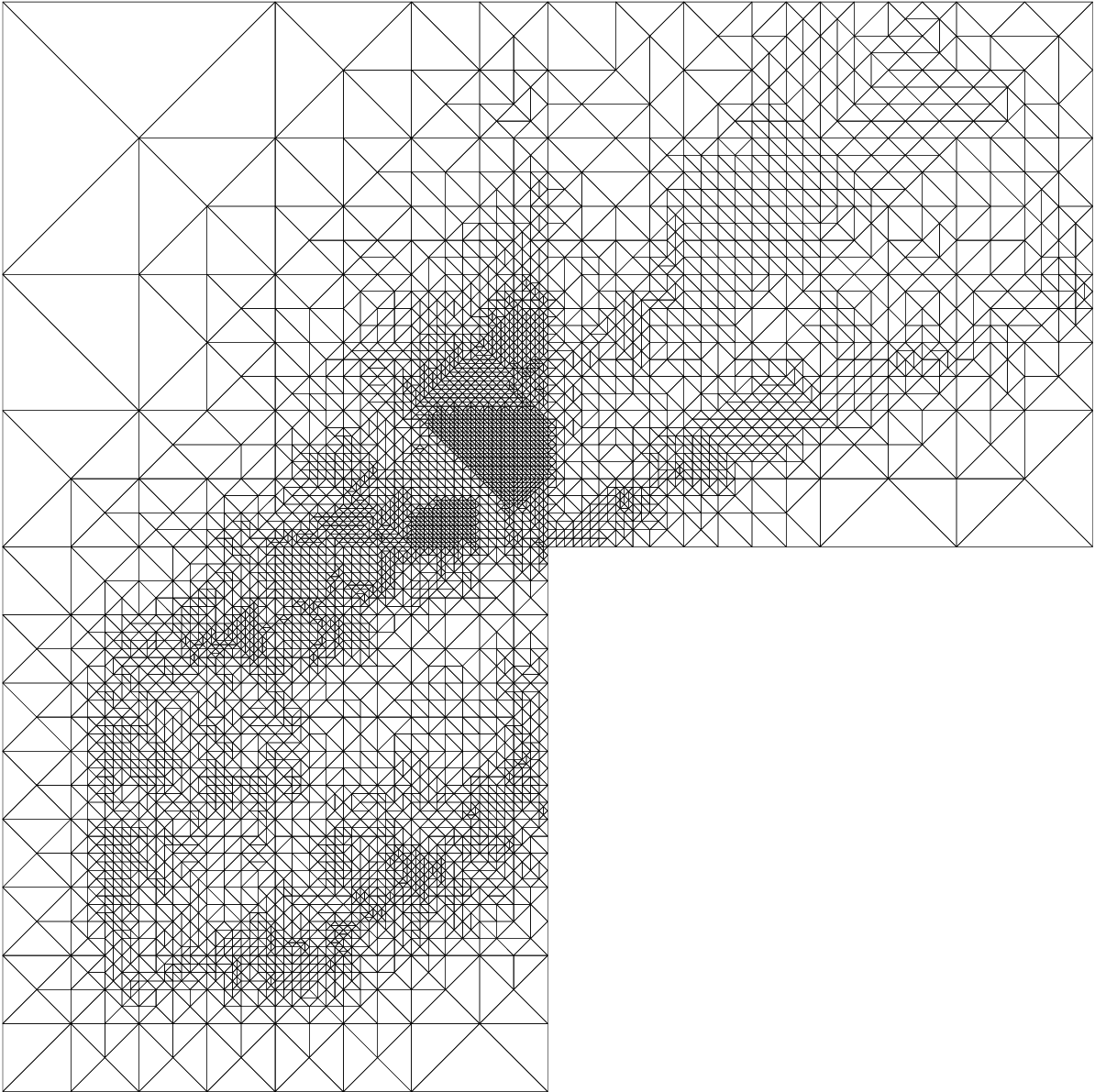}}
 \caption{\label{fig:practicalMesh}
 Two adaptively refined meshes $\TT^{(4)}$ and $\TT^{(6)}$ of the fifth and seventh step
 of the example in~\cref{subsec:bsp3}. The refinement mainly takes place along the convection direction.}
\end{figure} 

  \begin{figure}
 \centering
	\includegraphics[scale=0.6]{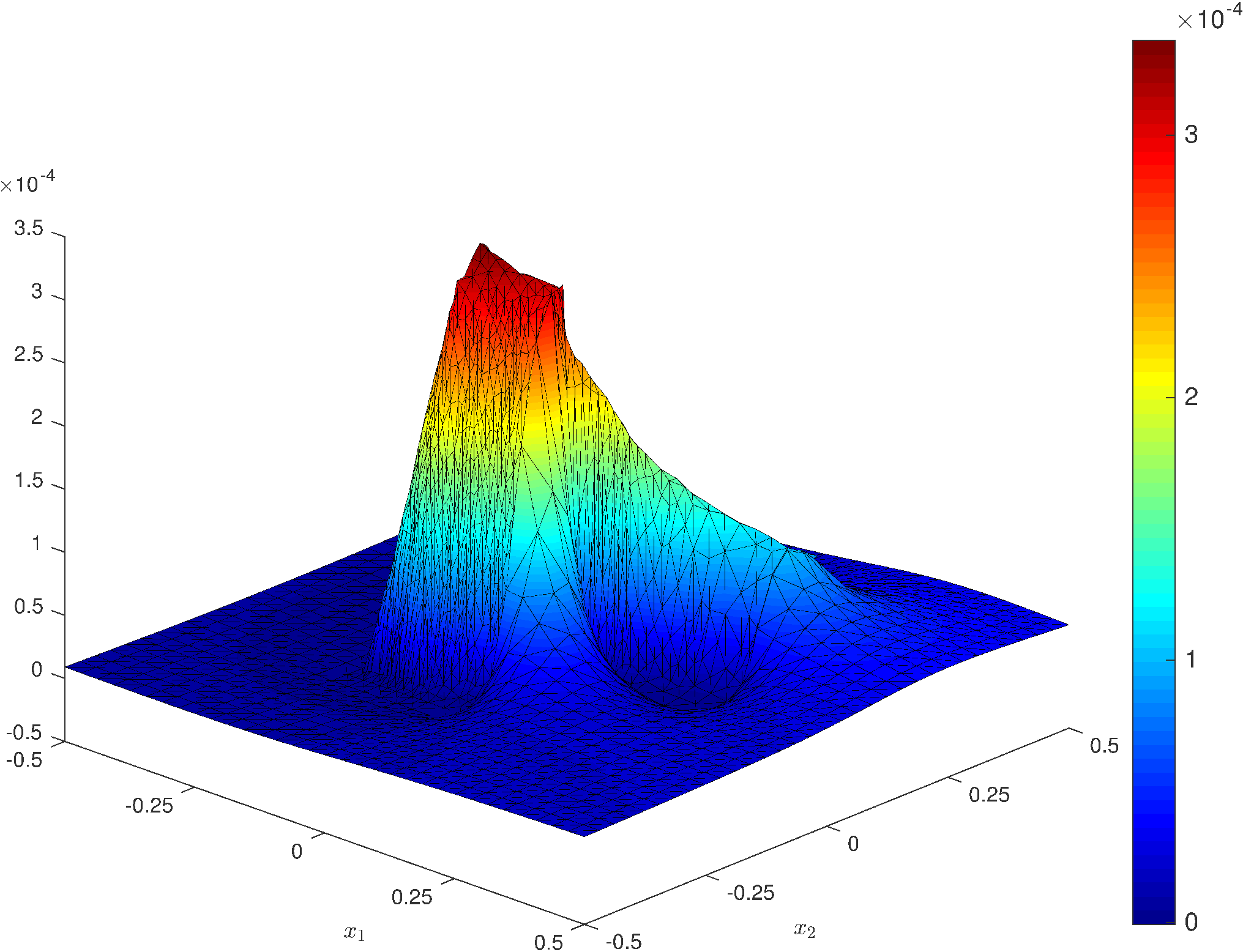}
 \caption{\label{fig:practicalSolution}
 The computed solution for the example in \cref{subsec:bsp3} on an
 adaptively generated mesh $\TT^{(4)}$ with 3471 elements.}
\end{figure}

\subsection{A more practical problem}
\label{subsec:bsp3}
For the third example we do not know an analytical solution of \cref{eq:model}.
Additionally, we replace the radiation condition~\cref{eq3:model} by
$u_e(x)=a_\infty+\OO(1/|x|)$ for $|x|\to \infty$. Thus we have
to assume the scaling condition (in two dimensions)
\begin{align*}
	\spe{\partial u_e/\partial\normal}{1}{\Gamma}=0,
\end{align*}
see~\cite{Erath:2013-1}.
The constant $a_\infty$ has to be added to the representation formula.
So we have the additional term $\spe{\psi_h}{a_\infty}{\Gamma}$ on the
left-hand side of~\eqref{eq2:fvm} and we add an
equation that ensures $\spe{1}{\phi_h}{\Gamma}=0$.
The domain will be the classical L-shaped domain as in the example in~\cref{subsec:bsp1}.
We fix the piecewise constant diffusion coefficients to 
\begin{align*}
 \alpha=\begin{cases} 0.5 \quad &\text{for}~ x_1 > 0, \\
10 \quad &\text{for}~ x_2\leq 0,\\
50 \quad &\text{else,} \end{cases}
\end{align*}
$\b=(15000,10000)^T$, and $\r=0.01$. The right-hand side will be
\begin{align*}
 f(x_1,x_2)=\begin{cases} 50 \quad &\text{for}~-0.2\leq x_1\leq -0.1,\quad -0.2\leq x_2\leq -0.05, \\
0 \quad &\text{else,} \end{cases}
\end{align*}
and the jumps $t_0$ and $u_0$ are set to zero.
This problem is again convection dominated. Therefore, we use the
full upwind stabilization \cref{eq:upwind}.
The convergence rate of the error estimator 
is plotted in~\cref{fig:practicalerror}. We observe 
a suboptimal convergence order $\OO(N^{-2/5})$ in the uniform case. However, we can again recover
the order $\OO(N^{-1/2})$ with our adaptive strategy. Note that since we have proven 
the reliability and the efficiency of our error estimator, these rates depict the convergence
behaviour of the error.
Adaptively generated meshes $\TT^{(4)}$ and $\TT^{(6)}$ from a start mesh with $48$ elements
are shown in~\cref{fig:practicalMesh}. The plots show
that the mesh is the finest along the direction of the convection.
Finally, in \cref{fig:practicalSolution} we see the interior and parts of the exterior
discrete solution. The interior transport problem influences the exterior part, which
describes a diffusion process, and the solution is continuous over the coupling boundary $\Gamma$.

\section{Conclusions}

This work provides an a~posteriori error estimator for the non-symmetric FVM-BEM coupling discretization
of~\cite{Erath:2015-2}.
The error estimator bounds the error from 
above and, under some restrictions on the mesh, also from below.
Additional assumptions even allow the construction of a robust error estimator, where the
upper bound is fully robust against variation of the model data. 
Note that the upper estimate only holds if the diffusion is above a certain (theoretical) bound.
The lower bound, however, additionally
depends on the P\'eclet number. 
The analysis relies on an ellipticity estimate in the energy (semi)norm and therefore differs from the
a~posteriori analysis of the three field FVM-BEM coupling in~\cite{Erath:2013-1}.
We think that this work and~\cite{Erath:2013-1} complete the residual based a~posteriori
error estimation theory for vertex-centered FVM-BEM couplings. Hence, it should be possible
to transfer the results directly to Bielak-MacCamy or the symmetric coupling approaches.

\section*{Acknowledgements}
The authors gratefully acknowledge G\"unther Of (TU Graz, Austria) for his valuable hints
to show Lemma~\ref{lem:robustellipt}.
The work of the second author is supported by the 'Excellence Initiative' of the German Federal and State Governments 
and the Graduate School of Computational Engineering at Technische Universitß\"at Darmstadt.

\bibliographystyle{alpha}

\bibliography{literature}

\newcommand{\etalchar}[1]{$^{#1}$}
\begin{thebibliography}{AEF{\etalchar{+}}13}

\bibitem[AEF{\etalchar{+}}13]{HILBERT:2013-1}
M.~Aurada, M.~Ebner, M.~Feischl, S.~Ferraz-Leite, T.~F{\"u}hrer, P.~Goldenits,
  M.~Karkulik, M.~Mayr, and D.~Praetorius.
\newblock {HILBERT} --- a {MATLAB} implementation of adaptive {2D-BEM}.
\newblock {\em Numerical Algorithms}, 67(1):1--32, 2013.

\bibitem[AFF{\etalchar{+}}13]{Aurada:2013-1}
M.~Aurada, M.~Feischl, T.~F{\"u}hrer, M.~Karkulik, J.~M. Melenk, and
  D.~Praetorius.
\newblock {Classical FEM-BEM coupling methods: nonlinearities, well-posedness,
  and adaptivity}.
\newblock {\em Computational Mechanics}, 51(4):399--419, 2013.

\bibitem[Car97]{Carstensen:1997-1}
C.~Carstensen.
\newblock {An a posteriori error estimate for a first-kind integral equation}.
\newblock {\em Math. Comp.}, 66(217):139--155, 1997.

\bibitem[Cia78]{Ciarlet:1978-book}
P.~G. Ciarlet.
\newblock {\em {The finite element method for elliptic problems}}.
\newblock North-Holland Publishing Co., Amsterdam-New York-Oxford, 1978.

\bibitem[Cl{\'e}75]{Clement:1975-1}
P.~Cl{\'e}ment.
\newblock {Approximation by finite element functions using local
  regularization}.
\newblock {\em Rev. Fran caise Automat. Informat. Recherche Op\'erationnelle
  S\'er. RAIRO Analyse Num\'erique}, 9(R-2):77--84, 1975.

\bibitem[CMS01]{Carstensen:2001-1}
C.~Carstensen, M.~Maischak, and E.~P. Stephan.
\newblock {A~posteriori error estimate and h-adaptive algorithm on surfaces for
  Symm's integral equation}.
\newblock {\em Numer. Math.}, 90(2):197--213, 2001.

\bibitem[Cos88]{Costabel:1988-1}
M.~Costabel.
\newblock {Boundary integral operators on Lipschitz domains: elementary
  results}.
\newblock {\em SIAM J. Math. Anal.}, 19(3):613--626, 1988.

\bibitem[D{\"{o}}r96]{Doerfler:1996-1}
W.~D{\"{o}}rfler.
\newblock A convergent adaptive algorithm for {P}oisson's equation.
\newblock {\em SIAM J. Numer. Anal.}, 33(3):1106--1124, 1996.

\bibitem[EOS15]{Erath:2015-2}
C.~Erath, G.~Of, and F.-J. Sayas.
\newblock {A non-symmetric coupling of the finite volume method and the
  boundary element method}.
\newblock {\em Preprint, arXiv:1509.00440}, 2015.

\bibitem[Era10]{Erath:2010-phd}
C.~Erath.
\newblock {\em {Coupling of the Finite Volume Method and the Boundary Element
  Method - Theory, Analysis, and Numerics}}.
\newblock PhD thesis, University of Ulm, April 2010.

\bibitem[Era12]{Erath:2012-1}
C.~Erath.
\newblock {Coupling of the finite volume element method and the boundary
  element method: an a priori convergence result}.
\newblock {\em SIAM Journal on Numerical Analysis}, 50(2):574--594, 2012.

\bibitem[Era13a]{Erath:2013-2}
C.~Erath.
\newblock {A new conservative numerical scheme for flow problems on
  unstructured grids and unbounded domains}.
\newblock {\em Journal of Computational Physics}, 245:476--492, 2013.

\bibitem[Era13b]{Erath:2013-1}
C.~Erath.
\newblock {A posteriori error estimates and adaptive mesh refinement for the
  coupling of the finite volume method and the boundary element method}.
\newblock {\em SIAM Journal on Numerical Analysis}, 51(3):1777--1804, 2013.

\bibitem[McL00]{McLean:2000-book}
W.~McLean.
\newblock {\em {Strongly elliptic systems and boundary integral equations}}.
\newblock Cambridge University Press, Cambridge, 2000.

\bibitem[OS13]{Of:2013-1}
G{\"u}nther Of and Olaf Steinbach.
\newblock {Is the one-equation coupling of finite and boundary element methods
  always stable?}
\newblock {\em Z. Angew. Math. und Mech.}, 93(6-7):476--484, 2013.

\bibitem[Pet02]{Petzoldt:2002-1}
M.~Petzoldt.
\newblock {A posteriori error estimators for elliptic equations with
  discontinuous coefficients}.
\newblock {\em Adv. Comput. Math.}, 16:47--75, 2002.

\bibitem[RST96]{Roos:1996-book}
H.~G. Roos, M.~Stynes, and L.~Tobiska.
\newblock {\em {Numerical methods for singularly perturbed differential
  equations}}, volume~24.
\newblock Springer, Berlin, Berlin, Heidelberg, 1996.

\bibitem[SW01]{Steinbach:2001-1}
O.~Steinbach and W.~L. Wendland.
\newblock {On C. Neumann's method for second-order elliptic systems in domains
  with non-smooth boundaries}.
\newblock {\em J. Math. Anal. Appl.}, 262(2):733--748, 2001.

\bibitem[Ver96]{Verfurth:book-1996}
R.~Verf{\"u}rth.
\newblock {\em {A Review of A Posteriori Error Estimation and Adaptive
  Mehs-Refinement Techniques}}.
\newblock Wiley-Teubner, Stuttgart, 1996.

\bibitem[Ver98]{Verfurth:1998-1}
R.~Verf{\"u}rth.
\newblock {A posteriori error estimators for convection-diffusion equations}.
\newblock {\em Numer. Math.}, 80(4):641--663, 1998.

\end{thebibliography}

\end{document}